\documentclass[11pt]{article}

\usepackage[margin=1in]{geometry}
\usepackage{amsmath,amssymb,amsfonts,amsthm,mathtools}
\usepackage[mathscr]{euscript}
\usepackage{enumitem}
\usepackage{hyperref}
\usepackage{microtype}
\usepackage{authblk}
\hypersetup{colorlinks=true,linkcolor=blue,citecolor=blue,urlcolor=blue}

\numberwithin{equation}{section}
\setlist{nosep}

\theoremstyle{plain}
\newtheorem{theorem}{Theorem}[section]
\newtheorem{lemma}[theorem]{Lemma}

\newtheorem{corollary}[theorem]{Corollary}
\newtheorem{conjecture}[theorem]{Conjecture}
\newtheorem{problem}[theorem]{Problem}

\theoremstyle{definition}
\newtheorem{definition}[theorem]{Definition}
\newtheorem{example}[theorem]{Example}

\theoremstyle{remark}
\newtheorem{remark}[theorem]{\textbf{Remark}}

\newcommand{\A}{\mathcal A}
\newcommand{\B}{\mathcal B}
\newcommand{\C}{\mathcal C}
\newcommand{\D}{\mathcal D}
\newcommand{\E}{\mathcal E}
\newcommand{\F}{\mathcal F}
\newcommand{\G}{\mathcal G}
\newcommand{\Hcal}{\mathcal H}
\newcommand{\I}{\mathcal I}
\newcommand{\Kcal}{\mathcal K}
\newcommand{\Pcal}{\mathcal P}
\newcommand{\R}{\mathcal R}
\newcommand{\T}{\mathcal T}
\newcommand{\W}{\mathcal W}
\newcommand{\Om}{\Omega}
\newcommand{\meetall}{\bigwedge}
\newcommand{\sbinom}[2]{\genfrac{[}{]}{0pt}{}{#1}{#2}}
\newcommand{\sbinomsep}[3]{\genfrac{[}{]}{0pt}{}{#1}{#2}_{#3}^{\mathrm{sep}}}
\newcommand{\eps}{\varepsilon}
\newcommand{\kk}{\kappa}
\DeclareMathOperator{\rk}{rk}

\title{Erd\H{o}s--Ko--Rado and Hilton--Milner Theorems  in the \\Partition Lattice}

\author[1]{Mengyu Cao\thanks{E-mail: \texttt{myucao@ruc.edu.cn}. Supported by the National Natural Science Foundation of China (12301431) and Beijing Natural Science Foundation (1262010).}}
\author[2]{Jiaqi Liao\thanks{E-mail: \texttt{liaojq19@ustb.edu.cn}.}}
\author[3]{Haixiang Zhang\thanks{E-mail: \texttt{zhang-hx22@mails.tsinghua.edu.cn}.}}

\affil[1]{\small Institute for Mathematical Sciences, Renmin University of China, Beijing 100086, China}
\affil[2]{\small School of Mathematics and Physics, University of Science and Technology Beijing, Beijing 100083, China}
\affil[3]{\small Department of Mathematical Sciences, Tsinghua University, Beijing 100084, China}

\date{}

\begin{document}
\maketitle

\begin{abstract}
Let $M_n=M(K_{n+1})$ be the graphic matroid of the complete graph, and let
$\F_k(M_n)$ be its rank-$k$ flats.  We study families
$\A\subseteq\F_k(M_n)$ satisfying $\rk(A\wedge B)\ge t$ for all
$A,B\in\A$.  For $t=1$, this problem is exactly equivalent to Czabarka's
partition-EKR conjecture, first introduced in print by P.~L. Erd\H{o}s and
L.~A. Sz\'ekely~\cite{ErdosSzekelyHigher}.  We prove the corresponding
Erd\H{o}s--Ko--Rado theorem in the explicit linear range $n+1\ge8k$, giving
a constant-factor advance toward
the conjectured sharp range $n\ge2k$.  For every fixed $t$, we further prove
an Erd\H{o}s--Ko--Rado theorem under an explicit condition of order
$O_t(k^2)$ on the block number $n+1-k$, with equality only for a full
$t$-star.  We also determine the largest nontrivial intersecting families
under an explicit $O(k^6)$ threshold and characterize the unique extremal
family up to isomorphism.
\end{abstract}



\section{Introduction}

The Erd\H{o}s--Ko--Rado theorem is a starting point of extremal set theory.
It states that, when $n\ge 2k$, an intersecting family of $k$-subsets of an
$n$-element set has size at most $\binom{n-1}{k-1}$; for $n>2k$, equality is
attained only by a star~\cite{EKR}.  Hilton and Milner subsequently determined
the largest intersecting $k$-uniform families having no common element
\cite{HiltonMilner}.  The corresponding $t$-intersection problem culminated in
the complete intersection theorem of Ahlswede and Khachatrian
\cite{AhlswedeKhachatrian}.  These results have motivated analogues for vector
spaces, permutations, matchings, partitions, and many other homogeneous
combinatorial structures.

Matroids, introduced by Whitney~\cite{Whitney}, provide a natural setting in
which independence, rank, closure, and intersection can be studied without a
specific linear or graphical representation; see~\cite{Oxley} for background.
The flats of a simple matroid form a geometric lattice.  For the graphic
matroid of a complete graph this lattice is precisely a partition lattice: a
flat is a disjoint union of complete graphs, or equivalently a set partition of
the vertex set.  This classical correspondence makes $M(K_{n+1})$ a
particularly natural test case for intersection theorems on matroid flats.

Let $S(N,b)$ denote the classical Stirling number of the second kind,
the number of partitions of an $N$-element set into $b$ nonempty blocks.
To avoid repeatedly translating between lattice rank and block number, we
use the binomial-style rank indexing
\[
 \sbinom{x}{y}:=S(x+1,x+1-y)
 \qquad(0\le y\le x).
\]
Thus the \emph{Stirling binomial} $\sbinom{x}{y}$ is not a new array: it is
exactly the classical Stirling number $S(x+1,x+1-y)$, indexed by the rank
$y$ of a partition on $x+1$ points.

Czabarka's partition-EKR conjecture was first introduced in print by
P.~L. Erd\H{o}s and L.~A. Sz\'ekely~\cite{ErdosSzekelyHigher}; their
acknowledgment credits \'Eva Czabarka with the conjecture.  Thus they first
presented the problem in the literature, while Czabarka was its originator.
In their paper, two partitions are called
\emph{intersecting in a pair} if some two ground-set elements lie in one
block of each partition.  Conjecture~1 there asserts that, if a $v$-set is
partitioned into $\ell$ blocks and $v\le2\ell-1$, then a pairwise
pair-intersecting family has size at most
$\sbinom{v-2}{v-1-\ell}$, attained by fixing a pair in one block.  In the
terminology now standard for partition problems, this is a partially
$2$-intersecting problem.  Meagher, Shirazi, and Stevens proved the canonical
extremal result for uniform profiles when the number of blocks is sufficiently
large~\cite{MeagherShiraziStevens}.  Kupavskii subsequently proved the
corresponding eventual theorem for every fixed profile, uniformly once its
largest block size is prescribed~\cite[Theorem~9]{KupavskiiPartitions}.

For fixed excess $d=v-\ell$, only finitely many nonsingleton profile shapes
occur and every block has size at most $d+1$.  Summing Kupavskii's fixed-profile
conclusion over these shapes shows that the conjectured inequality holds for
every fixed $d$ once $\ell$ is sufficiently large.  This is an eventual
consequence rather than a linear threshold: Czabarka's conjecture remains open
for arbitrary admissible pairs $(v,\ell)$ in the full proposed range
$v\le2\ell-1$, and the profile theorem supplies no threshold linear in $d$.

The relation with our $t=1$ problem is exact, rather than merely analogous.
Set
\[
 v=n+1,\qquad \ell=n+1-k.
\]
Then rank-$k$ flats of $M(K_{n+1})$ are precisely the $\ell$-block
partitions of a $v$-set, pair-intersection is the condition
$\rk(A\wedge B)\ge1$, and
\[
 v\le2\ell-1\quad\Longleftrightarrow\quad n\ge2k,
 \qquad
 \sbinom{v-2}{v-1-\ell}=\sbinom{n-1}{k-1}.
\]
The observation immediately following the conjecture, concerning
$v=2\ell$, likewise translates into the obstruction one step below the
conjectured range in our rank notation.

A different intersection relation requires two partitions to share entire
blocks.  This whole-block model includes the uniform-partition results of
Meagher and Moura~\cite{MeagherMoura}, the Hilton--Milner theorem of Ku and
Wong~\cite{KuWongPartitions}, and the recent sharp Erd\H{o}s--Ko--Rado and
Hilton--Milner results of Wen and Lv for $k$-partitions
\cite{WenLvKPartitions}; Kupavskii's work also treats several such
models~\cite{KupavskiiPartitions}.  In contrast, the present paper studies
pair-intersection: two partitions are intersecting when some pair of
ground-set elements lies in a common block of both partitions.  Under the
complete-graph matroid correspondence, this is exactly the condition that
the meet of the associated flats has positive rank.  The two lines of work
are therefore adjacent, but neither contains the other.

The proofs mainly combine a rank-sensitive form of spread approximation and
peeling with exact extension counts in the partition lattice.  The linear
$t=1$ argument additionally uses an occupancy-tail estimate to separate flats
by atomic weight and a singleton-switching comparison between a full star and
the family outside it.  For the Hilton--Milner theorem, localized rank-two
covers convert popular centres into universal covers, after which a finite
rank-two classification determines the extremal structure.

The modern $r$-spread technology grew out of the robust-sunflower advances
of Alweiss, Lovett, Wu, and Zhang and of Rao
\cite{AlweissLovettWuZhang,RaoSunflowers}, with a closely related threshold
formulation developed by Frankston, Kahn, Narayanan, and Park
\cite{FrankstonKahnNarayananPark}.  Kupavskii and Zakharov introduced the
systematic spread-approximation and peeling framework for
forbidden-intersection problems~\cite{KupavskiiZakharovSpread}; Wen and Lv
later combined it with covers and fingerprints~\cite{WenLvUnified}, while
Ihringer and Kupavskii developed a subspace-spread analogue
\cite{IhringerKupavskii}.  In the partition lattice, the spread exponent must
be lattice rank rather than the cardinality of the atom set: the three edges
of a triangle have cardinality three but closure rank two.  Our proofs use the
deterministic heavy-link extraction and peeling component of this framework.

Write $\Om_n:=\{0,1,\ldots,n\}$, regard $K_{n+1}$ as the complete graph
on $\Om_n$, and put $M_n:=M(K_{n+1})$.  Let $\F_k(M_n)$ denote the rank-$k$
flats of $M_n$.  We use $\rk$, $\wedge$, and $\vee$ for lattice rank, meet,
and join, respectively, in accordance with standard geometric-lattice
notation.
Throughout the paper, every family described as intersecting or
$t$-intersecting is assumed to be nonempty.  Consequently, every meet
$\meetall_{A\in\A}A$ appearing below is a meet over a nonempty family.
In the notation of this paper, the inequality in Czabarka's conjecture takes
the following particularly simple form.  We record with it the natural
uniqueness strengthening.

\begin{conjecture}[Czabarka's partition-EKR conjecture]
\label{conj:sharp-EKR-t1}
Let $k\ge2$ and $n\ge2k$.  If
$\A\subseteq\F_k(M_n)$ is intersecting, then
\[
 |\A|\le\sbinom{n-1}{k-1}.
\]
The displayed inequality is Czabarka's conjecture.  We further conjecture
that equality holds only for a full edge-star.
\end{conjecture}

Our first main result addresses this $t=1$ conjecture in an explicit linear
range.  Its proof combines rank spread approximation with a sparse--dense
decomposition according to atomic weight.

\begin{theorem}[Linear Erd\H{o}s--Ko--Rado theorem for $M_n$]
\label{thm:EKR-linear-t1}
Let $k\ge2$ and suppose that $n+1\ge8k$.
If $\A\subseteq\F_k(M_n)$ is intersecting, then
\[
 |\A|\le\sbinom{n-1}{k-1}.
\]
Equality holds if and only if $\A$ is a full edge-star.
\end{theorem}

Before Theorem~\ref{thm:EKR-linear-t1}, the strongest general conclusion
available for Czabarka's unrestricted fixed-excess problem was the
consequence of Kupavskii's fixed-profile theorem described above: for each
fixed $k$, the conjectured inequality holds once $n$ is sufficiently large,
but that result supplies no threshold linear in $k$
\cite[Theorem~9]{KupavskiiPartitions}.  Theorem~\ref{thm:EKR-linear-t1}
replaces this qualitative largeness hypothesis by the explicit condition
$n+1\ge8k$ and also determines the equality case.  It therefore upgrades
the previous eventual result to a genuinely linear theorem and gives an
explicit constant-factor approximation to the conjectured sharp range
$n\ge2k$.

We next treat general $t$-intersection.  For its statement, put
\[
  m_k:=\binom{k+1}{2},
  \qquad c_t:=2(t+2)^{2t}(2t)!.
\]
The following result gives a completely explicit quadratic threshold for
every fixed $t$.

\begin{theorem}[Erd\H{o}s--Ko--Rado theorem for $M_n$]
\label{thm:EKR}
Let $k>t\ge1$ and suppose that
\begin{equation}
\label{eq:EKR-explicit-threshold}
  n+1-k\ge c_t(m_k+1).
\end{equation}
If $\A\subseteq\F_k(M_n)$ is $t$-intersecting, then
\[
  |\A|\le \sbinom{n-t}{k-t}.
\]
Equality holds if and only if there is a $t$-flat $T$ such that
\[
  \A=\{A\in\F_k(M_n):T\le A\}.
\]
\end{theorem}

\begin{remark}
\label{rem:threshold-obstructions}
For $t\ge2$, the sufficient condition in
\eqref{eq:EKR-explicit-threshold} is $O_t(k^2)$ for fixed $t$.  For $t=1$,
Theorem~\ref{thm:EKR-linear-t1} gives the linear bound $n+1\ge8k$, while
Czabarka's conjecture, Conjecture~\ref{conj:sharp-EKR-t1}, predicts the
sharp range $n\ge2k$.
A direct construction, proved in Section~\ref{sec:t-ekrproof}, shows that any
threshold valid for all larger $n$ must be at least $2k-t+1$.
For $t\ge2$ this lower bound is not sufficient: a compatible-split family beats
a full $t$-star at the proposed endpoint when $k=t+1$; see
Corollary~\ref{cor:EKR-core-obstruction}.  The intrinsic $N\ge3k$ barrier of
the singleton-switching argument is explained in
Remark~\ref{rem:linear-method-limit}.
\end{remark}

For the nontrivial problem, fix distinct vertices $0,1$ and a $k$-set
$Y\subseteq\Om_n\setminus\{0,1\}$.  Let $a$ be the atom corresponding to the
edge $01$, and let $C_i$ be the rank-$k$ flat whose only nonsingleton block is
$\{i\}\cup Y$, for $i\in\{0,1\}$.  Define
\begin{equation}
\label{eq:Gconstruction-intro}
 \G_k(a,Y):=
 \{F\in\F_k(M_n):a\le F\text{ and }\rk(F\wedge C_0)\ge1\}
 \cup\{C_0,C_1\}.
\end{equation}
This family is intersecting and has no common atom.

Two families $\A,\A'\subseteq\F_k(M_n)$ are called \emph{isomorphic} if
there is a permutation $\sigma\in\operatorname{Sym}(\Om_n)$ such that
$\sigma(\A)=\A'$.  Every such permutation acts naturally as an automorphism
of $M_n=M(K_{n+1})$ and hence on its lattice of flats.

For integers $1\le r\le N$ and $1\le b\le N$, let
$\sbinomsep{N-1}{N-b}{r}$ denote the number of partitions of an $N$-element
set into $b$ blocks in which $r$ prescribed elements lie in distinct blocks.
The exact formula is given in Lemma~\ref{lem:separated-stirling}.  Set
\[
 \operatorname{HM}(n,k)
 :=\sbinom{n-1}{k-1}-\sbinomsep{n-1}{k-1}{k+1}+2.
\]

\begin{theorem}[Hilton--Milner theorem for $M_n$]
\label{thm:HM}
Let $k\ge3$, put $m=\binom{k+1}{2}$, and suppose that
\begin{equation}
\label{eq:HM-explicit-threshold}
  n+1-k\ge3\left(m^3+\binom m2\right).
\end{equation}
If $\A\subseteq\F_k(M_n)$ is intersecting and
$\rk(\meetall_{A\in\A}A)=0$, then
\[
  |\A|\le \operatorname{HM}(n,k).
\]
Equality holds if and only if $\A$ is isomorphic to $\G_k(a,Y)$.

For $k=2$ the same conclusion holds for every $n\ge3$, and
$\operatorname{HM}(n,2)=5$.
\end{theorem}

The presence of both $C_0$ and $C_1$ in
\eqref{eq:Gconstruction-intro} is essential.  If the fixed atom is disjoint
from the exceptional clique, only one exceptional rank-$k$ flat is available;
that construction is smaller by exactly one and is therefore not extremal.

The paper is organized as follows.  Section~\ref{sec:preliminaries}
collects the matroid and partition-lattice preliminaries.
Section~\ref{sec:rank-peeling} develops rank spreadness, peeling, and
localized rank-two covers.  Section~\ref{sec:proofs} proves the general
$t$-intersection theorem and the Hilton--Milner theorem.  The linear
Erd\H{o}s--Ko--Rado theorem is proved separately in
Section~\ref{sec:linear-ekrproof}.
Section~\ref{sec:further} discusses threshold obstructions and open problems.
Appendix~\ref{app:technical-estimates} contains the elementary estimates used
to close the explicit constants; all combinatorial and probabilistic
reductions remain in the main text.

\section{Matroid and partition-lattice preliminaries}
\label{sec:preliminaries}

\subsection{Basic definitions}

\begin{definition}
A \emph{matroid} is a pair $M=(E,\I)$, where $E$ is a finite set and
$\I\subseteq2^E$ satisfies:
\begin{enumerate}[label=\textup{(\arabic*)}]
\item $\varnothing\in\I$;
\item if $I\in\I$ and $I'\subseteq I$, then $I'\in\I$;
\item if $I_1,I_2\in\I$ and $|I_1|<|I_2|$, then some
      $e\in I_2\setminus I_1$ satisfies $I_1\cup\{e\}\in\I$.
\end{enumerate}
The members of $\I$ are the independent sets.  For $X\subseteq E$, its rank is
\[
  \rk_M(X)=\max\{|I|:I\subseteq X,\ I\in\I\}.
\]
A set $F\subseteq E$ is a \emph{flat} if
$\rk(F\cup\{e\})=\rk(F)+1$ for every $e\in E\setminus F$.
We write $\F(M)$ for the lattice of flats, ordered by inclusion, and
$\F_k(M)$ for its rank-$k$ level.
\end{definition}

The meet in $\F(M)$ is set intersection, while the join is
$F\vee F'=\operatorname{cl}(F\cup F')$.

\begin{example}
If $G=(V,E)$ is a finite graph and $\I$ is the family of forests in $G$, then
$(E,\I)$ is the graphic matroid $M(G)$.  Its rank function is
\[
  \rk_{M(G)}(X)=|V|-c(V,X),
\]
where $c(V,X)$ is the number of connected components of the spanning graph
$(V,X)$.
\end{example}

Recall that $K_{n+1}$ has vertex set $\Om_n$ and that
$M_n=M(K_{n+1})$.
For a partition $\pi$ of $\Om_n$, let $F_\pi$ be the union of the edge sets of
the complete graphs induced by the blocks of $\pi$.  Graphic closure fills in
each connected component, so $\pi\mapsto F_\pi$ identifies the partition
lattice of $\Om_n$, ordered by refinement, with $\F(M_n)$.  Under this
identification,
\[
  \rk(F_\pi)=n+1-|\pi|.
\]

For $X\subseteq\Om_n$ with $|X|\ge2$, write $\kk(X)$ for the flat whose only
nonsingleton block is $X$.  In particular,
$\eps_{xy}:=\kk(\{x,y\})$ is the atom corresponding to the edge $xy$.
For a flat $F$, define its set of atoms by
\[
  \tau(F):=\{\eps_{xy}:\eps_{xy}\le F\}.
\]
Under the graphic representation, $\tau(F)$ is simply the edge set of $F$.

\begin{lemma}
\label{lem:number-atoms}
If $F$ has rank $r$, then
\[
  |\tau(F)|\le \binom{r+1}{2}.
\]
Equality holds if and only if $F=\kk(X)$ for some $(r+1)$-set $X$.
\end{lemma}

\begin{proof}
Let the nonsingleton blocks of $F$ have sizes $s_1,\ldots,s_\ell$.  Then
\[
  \sum_{i=1}^{\ell}(s_i-1)=r,
  \qquad
  |\tau(F)|=\sum_{i=1}^{\ell}\binom{s_i}{2}.
\]
Moving one unit of size from a smaller nontrivial block to a larger one does
not decrease the second sum.  Hence it is maximized by a single block of size
$r+1$, giving $\binom{r+1}{2}$.  The equality statement follows from the same
strict convexity argument.
\end{proof}

\subsection{Counting flats and stars}

Under the partition-lattice identification,
\[
  |\F_k(M_n)|=\sbinom{n}{k}.
\]

\begin{lemma}[Separated Stirling numbers]
\label{lem:separated-stirling}
Let $1\le r\le N$ and $1\le b\le N$.  The number of partitions of an
$N$-element set into $b$ blocks in which $r$ prescribed elements lie in
distinct blocks is
\[
 \sbinomsep{N-1}{N-b}{r}
 =
 \begin{cases}
 0,& b<r,\\
 r^{N-r},& b=r,\\[2mm]
 \displaystyle\sum_{j=0}^{N-b}\binom{N-r}{j}r^j
 \sbinom{N-r-j-1}{N-b-j},& b>r.
 \end{cases}
\]
\end{lemma}

\begin{proof}
The count is zero when $b<r$.  If $b=r$, every remaining element is assigned
to one of the $r$ blocks containing the prescribed elements, giving
$r^{N-r}$ choices.  Suppose that $b>r$.  Choose the $j$ nonprescribed
elements that join those $r$ blocks and assign each of them to one of the
$r$ blocks.  The remaining $N-r-j$ elements must form the other $b-r$
blocks, which can be done in
$\sbinom{N-r-j-1}{N-b-j}$ ways.  Necessarily $0\le j\le N-b$, and summing
over $j$ gives the formula.
\end{proof}

\begin{lemma}
\label{lem:extension-count}
If $R\in\F_r(M_n)$ and $r\le k\le n$, then
\[
  |\{F\in\F_k(M_n):R\le F\}|
  =\sbinom{n-r}{k-r}.
\]
\end{lemma}

\begin{proof}
The partition associated with $R$ has $n+1-r$ blocks.  A flat $F$ contains
$R$ precisely when its partition is a coarsening of that partition.  To obtain
rank $k$, the $n+1-r$ blocks must be partitioned into $n+1-k$ nonempty groups.
\end{proof}

For later use, it is convenient to abbreviate
\[
  E_r=E_r(n,k):=\sbinom{n-r}{k-r}
  \qquad(0\le r\le k)
\]
and
\[
  b:=n+1-k.
\]
Thus $E_r$ is the number of rank-$k$ extensions of any fixed rank-$r$
flat, and $b$ is the number of blocks in every partition representing a
rank-$k$ flat.

\begin{lemma}
\label{lem:exact-extension-ratio}
For $0\le r<k$,
\begin{equation}
\label{eq:exact-extension-ratio}
  E_r=bE_{r+1}+\sbinom{n-r-1}{k-r}\ge bE_{r+1}.
\end{equation}
The inequality is strict when $b\ge2$.  In particular,
$E_{r+1}\le E_r/b$, and $E_s\le E_r$ whenever $r\le s\le k$.
\end{lemma}

\begin{proof}
Apply the standard recurrence
\[
  \sbinom{x}{y}
  =(x+1-y)\sbinom{x-1}{y-1}+\sbinom{x-1}{y}
\]
with $x=n-r$ and $y=k-r$.  If $b\ge2$, then the final Stirling binomial in
\eqref{eq:exact-extension-ratio} is positive.  The remaining assertions
follow immediately.
\end{proof}

\begin{lemma}
\label{lem:number-subflats}
If $A\in\F_k(M_n)$, then the number of rank-$t$ flats contained in $A$ is at
most $\sbinom{k}{t}$.
Equality is attained when $A$ is a rank-$k$ clique flat.
\end{lemma}

\begin{proof}
Let the nonsingleton blocks of $A$ be $X_1,\ldots,X_\ell$, with
$|X_i|=s_i$.  Thus
\[
  \sum_{i=1}^{\ell}(s_i-1)=k.
\]
Choose a distinguished point $x_i\in X_i$ for each $i$, and identify all the
$x_i$ to a single point.  The resulting ground set has
\[
  \sum_{i=1}^{\ell}s_i-(\ell-1)=k+1
\]
points.

A subflat $T\le A$ is obtained by partitioning each $X_i$ into smaller
blocks.  After the above identification, merge the blocks containing the
distinguished points and leave every other block unchanged.  If the
restrictions of $T$ to the $X_i$ have ranks $t_i$, the resulting partition
has rank $\sum_i t_i=\rk(T)$.  Moreover, the original restrictions are
recovered by restricting the resulting partition to the images of the
$X_i$, so this map is injective.  Rank-$t$ subflats of $A$ therefore inject
into rank-$t$ flats of $M(K_{k+1})$, of which there are
$\sbinom{k}{t}$.  A clique flat has all these subflats, proving equality.
\end{proof}

\begin{definition}
Let $\A\subseteq\F_k(M_n)$.
\begin{enumerate}[label=\textup{(\arabic*)}]
\item The family $\A$ is \emph{$t$-intersecting} if
      $\rk(A\wedge A')\ge t$ for all $A,A'\in\A$.
\item It is \emph{trivially $t$-intersecting} if
      $\rk(\meetall_{A\in\A}A)\ge t$.
\item If $T\in\F_t(M_n)$, the family
      $\E_k(T):=\{F\in\F_k(M_n):T\le F\}$ is the full $t$-star with centre
      $T$.
\end{enumerate}
\end{definition}

By Lemma~\ref{lem:extension-count},
$|\E_k(T)|=\sbinom{n-t}{k-t}$.

\section{Spread approximation and localized covers}
\label{sec:rank-peeling}

\subsection{Rank spreadness and peeling}

For a family $\D$ of flats and a flat $X$, write
\[
  \D[X]:=\{D\in\D:X\le D\}.
\]
Ordinary spreadness bounds links by a power of the cardinality increment.
For flats, cardinality is not the correct parameter: the three edges of a
triangle have cardinality three but closure rank two.  We therefore use the
following rank version, analogous to subspace spreadness
\cite{IhringerKupavskii}.  Only its deterministic maximal-link consequence is
needed below.

\begin{definition}
Let $r>1$.  A nonempty family $\D[X]$ is \emph{$r$-rank-spread above $X$}
if, for every flat $Y>X$,
\begin{equation}
\label{eq:rank-spread}
  |\D[Y]|<r^{-(\rk(Y)-\rk(X))}|\D[X]|.
\end{equation}
\end{definition}

The strict inequality is convenient in the peeling argument.  The following
maximal-link lemma is purely order-theoretic; it uses only additivity of rank
increments along a chain.

\begin{lemma}
\label{lem:rank-spread-extraction}
Suppose that every member of $\D[X]$ has rank at most $\rk(X)+d$.  If
$|\D[X]|>r^d$, then there is a flat $Y\ge X$ such that
$|\D[Y]|>1$ and $\D[Y]$ is $r$-rank-spread above $Y$.
\end{lemma}

\begin{proof}
Choose an inclusion-maximal $Y\ge X$ satisfying
\[
 |\D[Y]|\ge r^{-(\rk(Y)-\rk(X))}|\D[X]|.
\]
The displayed hypothesis implies $|\D[Y]|>1$, since otherwise
$|\D[X]|\le r^{\rk(Y)-\rk(X)}\le r^d$.  If $Z>Y$, maximality gives
\[
 |\D[Z]|<r^{-(\rk(Z)-\rk(X))}|\D[X]|
 \le r^{-(\rk(Z)-\rk(Y))}|\D[Y]|,
\]
which is precisely~\eqref{eq:rank-spread} above $Y$.
\end{proof}

We now formalize the peeling--simplification procedure of Kupavskii and
Zakharov~\cite{KupavskiiZakharovSpread} and its subspace version
\cite{IhringerKupavskii}.  For a family $\mathcal{S}$ of flats, put
\[
  \A[\mathcal{S}]:=\bigcup_{S\in\mathcal{S}}\A[S].
\]
Set $\C_k=\A$.  Inductively, suppose that $\C_i$ is a $t$-intersecting
antichain of flats of rank at most $i$ and that the already frozen layers are
$\W_{i+1},\ldots,\W_k$.  During stage $i$, call a flat $X$
\emph{admissible} for the current antichain $\C$ if
\begin{enumerate}[label=\textup{(\arabic*)}]
\item $\rk(X)\ge t$ and $X$ is a strict subflat of at least one member of
      $\C$; and
\item $X$ $t$-intersects every member of $\C\setminus\C[X]$.
\end{enumerate}
For an admissible $X$, replace
\[
 \C\quad\hbox{by}\quad(\C\setminus\C[X])\cup\{X\},
\]
and then immediately delete every member that strictly contains another
current member.  Thus the current family remains an antichain.  The
replacement preserves $t$-intersection: condition~(1) gives
$\rk(X)\ge t$, and condition~(2) checks $X$ against precisely the current
members that are not removed.  At every point of stage $i$ we also have the
coverage invariant
\[
 \A=\A[\C]\cup\bigcup_{j=i+1}^k\A[\W_j].
\]
Indeed, if $X\le C$, then $\A[C]\subseteq\A[X]$; the same observation shows
that pruning a larger centre does not change the covered union.  Every
current centre remains a subflat of an original member of $\A$.

The procedure terminates.  Indeed, the integer-valued potential
\[
 \sum_{C\in\C}\rk(C)
\]
strictly decreases at every replacement, and pruning can only decrease it
further.  Let $\T_i$ denote the terminal antichain at stage $i$, let
$\W_i$ be its rank-$i$ members, and put
$\C_{i-1}:=\T_i\setminus\W_i$.  We then continue with stage $i-1$.

For brevity put
\[
  m_i:=\binom{i+1}{2},\qquad r_i:=m_i+1.
\]

\begin{lemma}[Partition-lattice peeling]
\label{lem:partition-peeling}
Let $\A\subseteq\F_k(M_n)$ be $t$-intersecting.  For every
$i=k,k-1,\ldots,t+1$, the preceding procedure has the following properties.
\begin{enumerate}[label=\textup{(\arabic*)}]
\item The terminal family $\T_i=\C_{i-1}\cup\W_i$ is a
      $t$-intersecting antichain, and
\begin{equation}
\label{eq:peeling-cover-invariant}
 \A=\A[\C_{i-1}]\cup\bigcup_{j=i}^k\A[\W_j].
\end{equation}
\item There are no subfamily $\D\subseteq\T_i$ and flat $X$ with
      $\rk(X)\ge t$ for which $|\D[X]|>1$ and $\D[X]$ is
      $r_i$-rank-spread above $X$.
\item One has
\begin{equation}
\label{eq:Wi-bound}
  |\W_i|\le \sbinom{i}{t}r_i^{i-t}.
\end{equation}
\end{enumerate}
\end{lemma}

\begin{proof}
The state invariants established above show that $\T_i$ is a
$t$-intersecting antichain.  They also show, by reverse induction on the
stages, that the original family is covered exactly as in
\eqref{eq:peeling-cover-invariant}; this proves part~(1).

For part~(2), suppose to the contrary that $\D[X]$ is such a spread
subfamily.  Since $|\D[X]|>1$ and $\T_i$ is an antichain, $X$ is a strict
subflat of a current member.  If $X$ $t$-intersected every member of
$\T_i\setminus\T_i[X]$, it would be admissible, contrary to the definition
of $\T_i$.  Hence there is
$F\in\T_i\setminus\T_i[X]$ such that
\[
  \rk(F\wedge X)<t.
\]
For every $S\in\D[X]$, the $t$-intersection of $\T_i$ gives
$\rk(F\wedge S)\ge t$.  Because every flat in a geometric lattice is the
join of its atoms, $F\wedge S$ contains an atom $e$ not below $X$; otherwise
$F\wedge S\le X$, and hence $F\wedge S\le F\wedge X$, contradicting the two
rank inequalities.  The flat $X\vee e$ has rank $\rk(X)+1$ and lies below
$S$.  Since $\rk(F)\le i$, Lemma~\ref{lem:number-atoms} supplies at most
$m_i$ choices for $e$.  Thus the spread condition gives
\[
 |\D[X]|
 \le\sum_{e\in\tau(F)\setminus\tau(X)}|\D[X\vee e]|
 <\frac{m_i}{r_i}|\D[X]|,
\]
a contradiction.

For part~(3), there is nothing to prove when $\W_i=\varnothing$.
Otherwise fix $F\in\W_i$.  Every $S\in\W_i$ contains a rank-$t$
subflat of $F\wedge S$.  There are at most $\sbinom{i}{t}$ choices for this
subflat by Lemma~\ref{lem:number-subflats}; hence some rank-$t$ flat
$T\le F$ satisfies
\[
 |\W_i[T]|\ge |\W_i|/\sbinom{i}{t}.
\]
If $|\W_i[T]|>r_i^{i-t}$, Lemma~\ref{lem:rank-spread-extraction} produces a
forbidden $r_i$-rank-spread link of size greater than one.  Therefore
$|\W_i[T]|\le r_i^{i-t}$, which proves~\eqref{eq:Wi-bound}.
\end{proof}

The same procedure admits a fixed spread parameter when the atomic weights
of all current centres are uniformly bounded.  We record this version
separately because it is used in the linear argument.

\begin{lemma}[Uniform-parameter peeling]
\label{lem:uniform-parameter-peeling}
Let $q\ge0$ and $r>\max\{q,1\}$.  Suppose that
$\A\subseteq\F_k(M_n)$ is $t$-intersecting and
$|\tau(A)|\le q$ for every $A\in\A$.  In the peeling procedure above,
part~\textup{(2)} of Lemma~\ref{lem:partition-peeling} remains valid with
each $r_i$ replaced by the fixed parameter $r$, and every terminal layer
satisfies
\begin{equation}
\label{eq:uniform-parameter-Wi}
 |\W_i|\le \sbinom{i}{t}r^{i-t}.
\end{equation}
The cover invariant and the antichain invariant are unchanged.
\end{lemma}

\begin{proof}
Every current centre is a subflat of an original member and therefore has at
most $q$ atoms.  Repeat the proof of Lemma~\ref{lem:partition-peeling}(2).
If an $r$-rank-spread link $\D[X]$ survived at a terminal stage, failure of
admissibility would provide a current flat $F$ with
$\rk(F\wedge X)<t$.  The same atomicity argument gives
\[
 |\D[X]|
 \le\sum_{e\in\tau(F)\setminus\tau(X)}|\D[X\vee e]|
 <\frac{q}{r}|\D[X]|,
\]
which is impossible because $r>q$.  The proof of part~(3), now using
Lemma~\ref{lem:rank-spread-extraction} with the fixed parameter $r$, gives
\eqref{eq:uniform-parameter-Wi}.  The state and coverage invariants do not
depend on the value of the spread parameter.
\end{proof}

Stopping the procedure at rank $s+t$ gives the following form used below.

\begin{corollary}[Peeling decomposition]
\label{cor:peeling-decomposition}
For $0\le s<k-t$, there is a $t$-intersecting family $\C_{s+t}$ of flats of
rank at most $s+t$ such that
\[
 \A=\A[\C_{s+t}]\cup\R_s,
 \qquad
 |\R_s|\le
 \sum_{i=s+t+1}^k \sbinom{i}{t}r_i^{i-t}E_i.
\]
Here one may take $\R_s=\A\setminus\A[\C_{s+t}]$; by the cover invariant it
is contained in the union of the extension families assigned to the peeled
layers $\W_{s+t+1},\ldots,\W_k$.
\end{corollary}

\begin{proof}
Use~\eqref{eq:peeling-cover-invariant}, then bound the extensions of each
$W\in\W_i$ by $E_i$ and apply~\eqref{eq:Wi-bound}.
\end{proof}

\subsection{Localized rank-two covers}

We first record the finite rank-two configuration used below.

\begin{lemma}[Nontrivial rank-two families]
\label{lem:rank-two-five}
Let $\C\subseteq\F_2(M_n)$ be intersecting and suppose that its members have
no common atom.  Then $|\C|\le5$.  Equality holds only, up to relabelling, for
\[
 \{\eps_{12}\vee\eps_{34},\
   \kk(\{1,2,3\}),\kk(\{1,2,4\}),
   \kk(\{1,3,4\}),\kk(\{2,3,4\})\}.
\]
\end{lemma}

\begin{proof}
A rank-two flat is either a triangle $\kk(\{x,y,z\})$ or the join of two
disjoint atoms, corresponding to a two-edge matching.

If every member is a triangle, identify it with its three graph vertices.
Two members intersect as flats exactly when the corresponding triples share
two vertices.  Take two distinct triples, say $123$ and $124$.  If the edge
$12$ is not common to the whole family, a further triple must be either $134$
or $234$; after these choices there are no other possibilities.  Hence this
case has at most four members.

Now suppose that $M\in\C$ is a matching flat with
$\tau(M)=\{x,y\}$.  Every other member contains $x$ or $y$, but not both.
Let $\Pcal$ be the members containing $x$ and let $\T$ be those containing
$y$, with $M$ omitted from both.  Nontriviality makes both subfamilies
nonempty.  The petals
\[
 \{\tau(P)\setminus\{x\}:P\in\Pcal\}
 \quad\text{and}\quad
 \{\tau(T)\setminus\{y\}:T\in\T\}
\]
are collections of nonempty sets of size at most two.  Within either
collection the petals are pairwise disjoint, while every petal in the first
collection meets every petal in the second.  Hence
$|\Pcal|,|\T|\le2$, and $|\C|\le1+2+2=5$.

At equality all four petals have size two.  Relabel the two disjoint edges of
$M$ as $12$ and $34$.  The two members through $12$ are triangles
$\kk(\{1,2,u\})$ and $\kk(\{1,2,v\})$, and the two through $34$ are
$\kk(\{3,4,p\})$ and $\kk(\{3,4,q\})$.  Cross-intersection forces
$\{u,v\}=\{3,4\}$ and $\{p,q\}=\{1,2\}$, giving exactly the displayed
family.
\end{proof}

Throughout this subsection, $k\ge3$, $\A\subseteq\F_k(M_n)$ is a nontrivial
intersecting family, and
\[
  m:=\binom{k+1}{2}.
\]
Define
\[
 \Psi_k(b):=\sum_{i=3}^k
 \frac{m_i(m_i+1)^{i-1}}{b^{i-2}}.
\]

\begin{lemma}[Rank-two peeling approximation]
\label{lem:rank-two-peeling}
There is an intersecting antichain $\C\subseteq\F_1(M_n)\cup\F_2(M_n)$ and
a remainder $\R\subseteq\A$ such that
\begin{equation}
\label{eq:rank-two-peeling-cover}
 \A=\A[\C]\cup\R,
 \qquad |\R|\le\Psi_k(b)E_2.
\end{equation}
If $\C$ has no common atom, then $\C\subseteq\F_2(M_n)$ and $|\C|\le5$.
If $\C$ has a common atom $a$, every member of $\A\setminus\A[a]$ lies in
$\R$.
\end{lemma}

\begin{proof}
Run the peeling procedure through stage two, and let
\[
 \C:=\T_2=\C_1\cup\W_2,
 \qquad
 \R:=\A\setminus\A[\C].
\]
The cover invariant in Lemma~\ref{lem:partition-peeling}(1) gives
\[
 \A=\A[\C]\cup\R,
 \qquad
 \R\subseteq\bigcup_{i=3}^k\A[\W_i].
\]
Since $\sbinom{i}{1}=m_i$, Lemma~\ref{lem:partition-peeling}(3) and repeated use
of Lemma~\ref{lem:exact-extension-ratio} yield
\[
 |\R|\le\sum_{i=3}^k m_i(m_i+1)^{i-1}E_i
 \le \Psi_k(b)E_2,
\]
which proves~\eqref{eq:rank-two-peeling-cover}.

The terminal family $\C=\T_2$ is an intersecting antichain of flats of rank
at most two.  Suppose first that its members have a common atom $a$.  If
$\C\ne\{a\}$, then $a$ is a strict subflat of at least one current member,
and every member outside $\C[a]$ is absent.  Thus $a$ is admissible at stage
two, contradicting the terminality of $\T_2$.  Hence $\C=\{a\}$.  In this
case $\A[\C]=\A[a]$, so every member of $\A\setminus\A[a]$ lies in $\R$.

If $\C$ has no common atom, it contains no rank-one flat: an atom belonging
to an intersecting antichain forces every other member to contain that atom.
Consequently $\C\subseteq\F_2(M_n)$, and
Lemma~\ref{lem:rank-two-five} gives $|\C|\le5$.
\end{proof}

Fix an atom $a$ and a member $C\in\A$ with $a\nleq C$.  The analogue of the
Wen--Lv cover family is
\[
 \Kcal(a,C):=\{a\vee e:e\in\tau(C)\}\subseteq\F_2(M_n).
\]
It has at most $m$ members, and it covers $\A[a]$: if $F\in\A[a]$, choose an
atom $e\le F\wedge C$ and observe that $a\vee e\le F$.

\begin{lemma}[Localized high covers]
\label{lem:localized-high-covers}
Let
\begin{equation}
\label{eq:localized-high-family}
 \Hcal(a,C):=\{B\in\Kcal(a,C):|\A[B]|>m^2E_3\}.
\end{equation}
Every $B\in\Hcal(a,C)$ intersects every member of $\A$.  Moreover,
\begin{equation}
\label{eq:localized-cover-bound}
 |\A[a]|\le |\Hcal(a,C)|E_2+m^3E_3.
\end{equation}
If $|\Hcal(a,C)|=m$, then $C$ is a rank-$k$ clique flat and
\[
 \Hcal(a,C)=\Kcal(a,C)=\{a\vee e:e\in\tau(C)\}.
\]
\end{lemma}

\begin{proof}
Fix $B\in\Hcal(a,C)$.  Greedily choose $m+1$ members
$F_1,\ldots,F_{m+1}$ of $\A[B]$ with
\[
 \tau(F_i)\cap\tau(F_j)=\tau(B)\qquad(i\ne j).
\]
After $j\le m$ choices, the union of the old petals contains at most
$jm\le m^2$ atoms.  For each petal atom $x$, the flat $B\vee x$ has rank
three and has $E_3$ rank-$k$ extensions.  Thus at most $m^2E_3$ members of
$\A[B]$ meet an old petal, and the strict inequality in
\eqref{eq:localized-high-family} permits the next choice.

If some $A\in\A$ did not intersect $B$, then, in order to intersect all
$F_i$, the atom set $\tau(A)$ would contain one atom from each of the $m+1$
pairwise disjoint petals.  This contradicts $|\tau(A)|\le m$.  Hence every
high $B$ is a $1$-cover of $\A$.

The at most $m$ low candidates contribute at most
$m\cdot m^2E_3=m^3E_3$, while each high candidate has at most $E_2$
extensions.  This proves~\eqref{eq:localized-cover-bound}.  Finally,
$|\Kcal(a,C)|\le|\tau(C)|\le m$.  Equality in
$|\Hcal(a,C)|\le|\Kcal(a,C)|\le|\tau(C)|\le m$ forces
$|\tau(C)|=m$, so Lemma~\ref{lem:number-atoms} gives the last assertion.
\end{proof}

\begin{lemma}
\label{lem:numerical-remainder}
If $k\ge3$ and
\[
 b\ge3\left(m^3+\binom m2\right),
\]
then
\begin{equation}
\label{eq:numerical-HM-gap}
 \frac{m^3+\binom m2}{b}+\Psi_k(b)<1.
\end{equation}
\end{lemma}

\begin{proof}
Put $Q=m+1$.  For $3\le i\le k$, one has $m_i\le m$ and
$m_i+1\le Q$.  Since $m\ge6$, also $Q\le7m/6$.  The hypothesis implies
$b\ge3m^3$, and hence
\begin{align*}
 \Psi_k(b)
 \le \frac{mQ^2}{b}
       \sum_{j=0}^{k-3}\left(\frac Qb\right)^j<\frac{49}{108}\left(1-\frac7{18m^2}\right)^{-1}
 <\frac12.
\end{align*}
The first term in~\eqref{eq:numerical-HM-gap} is at most $1/3$, so the sum is
strictly smaller than $5/6$.
\end{proof}

\begin{remark}
The proof of Theorem~\ref{thm:HM} below uses the closed-form hypothesis only
through~\eqref{eq:numerical-HM-gap}.  Thus
\[
 \frac{m^3+\binom m2}{b}+\Psi_k(b)<1
\]
is a sharper, directly computable sufficient condition for the same
conclusion.
\end{remark}
\section{The general EKR and Hilton--Milner theorems}
\label{sec:proofs}

\subsection{The general \texorpdfstring{$t$}{t}-intersection theorem}
\label{sec:t-ekrproof}

The rank-peeling proof is most naturally stated through the following
explicit dimensionless quantity:
\[
 \Phi_{k,t}(b):=
 \frac{m_k}{b}+
 \sum_{i=t+1}^k \sbinom{i}{t}
       \left(\frac{m_i+1}{b}\right)^{i-t}.
\]

\begin{lemma}
\label{lem:rank-spread-EKR}
If $\Phi_{k,t}(b)<1$, then every nontrivial $t$-intersecting family
$\A\subseteq\F_k(M_n)$ has size strictly smaller than $E_t$.
Moreover,
\begin{equation}
\label{eq:quadratic-EKR-sufficient}
 b\ge c_t(m_k+1)
 \quad\Longrightarrow\quad
 \Phi_{k,t}(b)<1.
\end{equation}
\end{lemma}

\begin{proof}
Apply Corollary~\ref{cor:peeling-decomposition} with $s=0$.  A
$t$-intersecting family of flats of rank at most $t$ is either empty or
consists of a single rank-$t$ flat, say $T$.  Hence
\[
 \A=\A[T]\cup\R_0,
 \qquad
 |\R_0|\le\sum_{i=t+1}^k \sbinom{i}{t}r_i^{i-t}E_i,
\]
where the first term is omitted if the core is empty.

Suppose that the core is $\{T\}$ and that $\A$ is nontrivial.  Choose
$A_0\in\A$ with $T\nleq A_0$.  If $F\in\A[T]$, then
$\rk(F\wedge A_0)\ge t$, whereas $\rk(T\wedge A_0)<t$.  We claim that
$F\wedge A_0$ contains an atom $e\nleq T$.  Otherwise every atom below
$F\wedge A_0$ would lie below $T$.  Since a flat in a geometric lattice is
the join of its atoms, this would imply
\[
 F\wedge A_0\le T,
 \qquad\text{and hence}\qquad
 F\wedge A_0\le T\wedge A_0,
\]
contradicting the two rank inequalities.  Thus such an atom $e$ exists.
Consequently $T\vee e\le F$ and $\rk(T\vee e)=t+1$.
Lemma~\ref{lem:number-atoms} gives at most $m_k$ choices for $e$, and
therefore
\[
  |\A[T]|\le m_k E_{t+1}.
\]
Repeated use of Lemma~\ref{lem:exact-extension-ratio} gives
$E_i/E_t\le b^{-(i-t)}$.  The last two displays, also covering the empty-core
case, yield
\[
 \frac{|\A|}{E_t}
 \le \frac{m_k}{b}
 +\sum_{i=t+1}^k \sbinom{i}{t}
       \left(\frac{r_i}{b}\right)^{i-t}
 =\Phi_{k,t}(b).
\]
This proves the first assertion.

It remains to verify~\eqref{eq:quadratic-EKR-sufficient}.  Put
$Q=m_k+1$.  Clearly $r_i\le Q$.  Also
\begin{equation}
\label{eq:subflat-polynomial-bound}
 \sbinom{i}{t}\le\binom{m_i}{t}\le(i+1)^{2t}.
\end{equation}
Indeed, fix a total order on the $m_i$ edges of $K_{i+1}$ and assign to
each rank-$t$ flat its lexicographically first $t$-edge forest basis.  The
closure of the chosen forest recovers the flat, so this assignment injects
the rank-$t$ flats into the $t$-subsets of the $m_i$ edges.  This proves the
first inequality in
\eqref{eq:subflat-polynomial-bound}.  Lemma~\ref{lem:appendix-Phi-bound} applies
this polynomial estimate and the binomial series to the hypothesis
$b\ge c_t Q$, giving
\[
 \sum_{i=t+1}^k \sbinom{i}{t}
       \left(\frac{r_i}{b}\right)^{i-t}\le\frac6{11},
 \qquad
 \frac{m_k}{b}<\frac1{36}.
\]
Therefore $\Phi_{k,t}(b)<6/11+1/36<1$, proving
\eqref{eq:quadratic-EKR-sufficient}.
\end{proof}

\begin{proof}[Proof of Theorem~\ref{thm:EKR}]
If $\A$ is trivially $t$-intersecting, choose a $t$-flat
$T\le\meetall_{A\in\A}A$.  Then
\[
  \A\subseteq\E_k(T),
  \qquad
  |\A|\le\sbinom{n-t}{k-t}.
\]
Suppose that $\A$ is nontrivial.  Since the hypothesis is
$b\ge c_t(m_k+1)$,
Lemma~\ref{lem:rank-spread-EKR} gives $|\A|<E_t$.  Thus every nontrivial
family satisfies the desired strict inequality.  Since
$E_t=\sbinom{n-t}{k-t}$, the upper bound follows.
Thus equality is possible only in the trivial case, where equality in the
containment $\A\subseteq\E_k(T)$ forces $\A=\E_k(T)$.  Conversely, every
full $t$-star is $t$-intersecting and has the asserted size by
Lemma~\ref{lem:extension-count}.
\end{proof}

\begin{proof}[Proof of the lower-bound assertion in
Remark~\ref{rem:threshold-obstructions}]
It is enough to construct a counterexample at $n=2k-t$; then
$b=n+1-k=k-t+1$ and $n\ge k+1$ because $k>t$.
Choose subsets $Z,Y\subseteq\Om_n$ and distinct vertices
$0,1\in\Om_n\setminus(Z\cup Y)$ such that
\[
 |Z|=t,\qquad |Y|=k-t+1,\qquad |Z\cap Y|=1.
\]
This uses exactly $k+2$ vertices and is therefore possible.  Let $R$ be
the flat whose only possibly nonsingleton block is $Z$; thus $R$ is the
bottom flat when $t=1$, and $R=\kk(Z)$ when $t\ge2$.  Put
\[
 T=R\vee\eps_{01},\qquad
 C_i=\kk(Z\cup Y\cup\{i\})\quad(i=0,1).
\]
Then $\rk(R)=t-1$, $\rk(T)=t$, $\rk(C_i)=k$, and
\[
  C_0\wedge C_1=\kk(Z\cup Y)
\]
has rank $k-1\ge t$.

Consider
\[
  \A:=\E_k(T)\cup\{C_0,C_1\}.
\]
We claim that every $F\in\E_k(T)$ $t$-intersects both exceptional flats.
Contract $T$.  Among the resulting points, distinguish the point
corresponding to the block $Z$, the point corresponding to the block
$\{0,1\}$, and the $k-t$ points in $Y\setminus Z$.  There are
$k-t+2$ distinguished points, whereas the partition corresponding to $F$
has
\[
  b=n+1-k\le k-t+1
\]
blocks.  Hence two distinguished points lie in a common block.

If neither of these points is the image of $\{0,1\}$, the resulting extra
atom lies in both $C_0$ and $C_1$.  If one of them is the image of
$\{0,1\}$, closure in $F$ supplies both edges $0x$ and $1x$ for the other
point $x$; one is contained in $C_0$ and the other in $C_1$.  In either case
$F\wedge C_i$ contains $R$ together with an atom outside $R$, so
\[
  \rk(F\wedge C_i)\ge t\qquad(i=0,1).
\]
Thus $\A$ is $t$-intersecting.  Neither $C_i$ contains the edge $01$, so the
two exceptional flats lie outside $\E_k(T)$ and
\[
  |\A|=|\E_k(T)|+2=\sbinom{n-t}{k-t}+2.
\]
\end{proof}

\subsection{The Hilton--Milner theorem}

The following classification is the structural heart of the proof.  Recall
$m=\binom{k+1}{2}$.

\begin{lemma}
\label{lem:core-classification}
Let $a=\eps_{uv}$ be an atom, and let
$B_1,\ldots,B_m$ be distinct rank-two flats containing $a$.  Suppose that a
flat $C$ of rank at most $k$ does not contain $a$ but intersects every $B_i$.
Then $\rk(C)=k$, $C=\kk(X)$ for a $(k+1)$-set $X$, and, after reindexing,
\[
  \{B_1,\ldots,B_m\}=\{a\vee e:e\in\tau(C)\}.
\]
Moreover, exactly one of the following occurs.
\begin{enumerate}[label=\textup{(\arabic*)}]
\item $X\cap\{u,v\}=\varnothing$.  The only flat of rank at most $k$
      that avoids $a$ and intersects every $B_i$ is $C$.
\item $|X\cap\{u,v\}|=1$.  Writing, say,
      $X=\{u\}\cup Y$, the only such flats are
      $\kk(\{u\}\cup Y)$ and $\kk(\{v\}\cup Y)$.
\end{enumerate}
\end{lemma}

\begin{proof}
The flats $B_i$ form a sunflower with kernel $a$: for distinct rank-two flats
through $a$, the petals
$\tau(B_i)\setminus\{a\}$ are nonempty and pairwise disjoint.  Since $C$ does
not contain $a$ but intersects every $B_i$, it contains at least one atom from
each of the $m$ petals.  Lemma~\ref{lem:number-atoms} gives
\[
  m\le|\tau(C)|\le\binom{\rk(C)+1}{2}\le\binom{k+1}{2}=m.
\]
All inequalities are equalities.  Thus $\rk(C)=k$, $C=\kk(X)$ for a
$(k+1)$-set $X$, and every petal contains exactly one atom of $C$.  For
$e\in\tau(C)$, the unique $B_i$ whose petal contains $e$ is the rank-two flat
$a\vee e$.

Because $a\nleq C$, the two endpoints $u,v$ do not both lie in $X$.  This
leaves cases (1) and (2).  Let $D$ be any further flat of rank at most $k$
that avoids $a$ and intersects every $B_i$.  Repeating the preceding argument
shows that $D$ is a rank-$k$ clique flat selecting exactly one atom from every
petal.

In case (1), every petal is the singleton $\{e\}$ for an edge
$e\in\binom{X}{2}$, so $D=C$.  In case (2), write $X=\{u\}\cup Y$.  The
petals belonging to the edges inside $Y$ are singletons, whereas the petal
belonging to $uy$ is $\{uy,vy\}$ for each $y\in Y$.  Hence $D$ contains all
edges inside $Y$ and chooses one of $uy,vy$ for every $y$.  A rank-$k$ clique
flat must make the same choice for all $y$.  Indeed, if it contained $uy$
and $vy'$ for distinct $y,y'\in Y$, then the clique on $Y$ would place
$u,v$, and all of $Y$ in one block.  That block has rank $k+1$ and also
forces $a\le D$, either of which is impossible.  The two possibilities are
therefore $\kk(\{u\}\cup Y)$ and $\kk(\{v\}\cup Y)$.
\end{proof}

\medskip
\noindent\textbf{The two canonical constructions.}

For case (1), let $a=\eps_{01}$ and let
$X\subseteq\Om_n\setminus\{0,1\}$ have size $k+1$.  Put $C=\kk(X)$ and
\[
 \mathscr F_k(a,X):=
 \{F\in\F_k(M_n):a\le F,\ \rk(F\wedge C)\ge1\}\cup\{C\}.
\]
For case (2), let $Y\subseteq\Om_n\setminus\{0,1\}$ have size $k$, put
$C_i=\kk(\{i\}\cup Y)$, and define $\G_k(a,Y)$ as in
\eqref{eq:Gconstruction-intro}.

\begin{lemma}[Canonical Hilton--Milner candidates]
\label{lem:canonical-HM-candidates}
Let $k\ge2$.  Whenever the relevant ground-set choices exist,
$\mathscr F_k(a,X)$ and $\G_k(a,Y)$ are nontrivial intersecting families.
Moreover, the first formula below holds for $n\ge k+2$, and the second for
$n\ge k+1$:
\begin{align*}
 |\mathscr F_k(a,X)|
 &=\sbinom{n-1}{k-1}-\sbinomsep{n-1}{k-1}{k+1}+1,\\
 |\G_k(a,Y)|
 &=\sbinom{n-1}{k-1}-\sbinomsep{n-1}{k-1}{k+1}+2.
\end{align*}
Hence, whenever $n\ge k+2$,
\[
  |\G_k(a,Y)|=|\mathscr F_k(a,X)|+1.
\]
\end{lemma}

\begin{proof}
We first verify the structural assertions.  In
$\mathscr F_k(a,X)$, any two members containing $a$ intersect at $a$, and
each of them intersects the exceptional clique $C=\kk(X)$ by definition.
Thus the family is intersecting.  It is nontrivial because $C$ avoids $a$,
and no atom of $C$ is common to the remaining members.  Indeed, given an
edge $e=\eps_{xy}\le C$, let $F_e$ be the flat whose nonsingleton blocks are
$\{0,1\}$ and $X\setminus\{y\}$.  These blocks contribute ranks $1$ and
$k-1$, respectively, so $\rk(F_e)=k$.  Moreover, $a\le F_e$,
$\rk(F_e\wedge C)\ge1$, and $e\nleq F_e$.  Hence no atom belongs to every
member of $\mathscr F_k(a,X)$.

The members of $\G_k(a,Y)$ containing $a$ intersect one another at $a$ and
meet $C_0$ by definition.  If such a member meets $C_0$ in an edge inside
$Y$, it also meets $C_1$ in that edge; if it meets $C_0$ in an edge $0y$,
closure together with $01$ supplies the edge $1y$, so it meets $C_1$.
Finally, $C_0\wedge C_1=\kk(Y)$ has rank $k-1\ge1$.  Thus $\G_k(a,Y)$ is
intersecting.  To see that it is nontrivial, any atom common to the whole
family would have the form $e=\eps_{yz}$ for distinct $y,z\in Y$.  Let
$F_y$ be the flat whose nonsingleton blocks are $\{0,1,y\}$ and
$Y\setminus\{y\}$, omitting the second block when it is a singleton.  Then
\[
 \rk(F_y)=2+(k-2)=k,
\]
$a\le F_y$, and $F_y$ meets $C_0$ in $\eps_{0y}$, while
$e\nleq F_y$.  Hence $\G_k(a,Y)$ has no common atom.

It remains to count the two families.  Contract the atom $a$, or equivalently
identify its endpoints.  A rank-$k$ flat containing $a$ then corresponds to
a partition of an $n$-element set into $n+1-k$ blocks.  In either
construction, the exceptional clique determines $k+1$ distinguished points
after this identification.  A flat through $a$ meets that clique if and only
if the distinguished points are not all in distinct blocks.  The two formulas
therefore follow from Lemma~\ref{lem:separated-stirling}; their difference is
one.  The five-member family in Lemma~\ref{lem:rank-two-five} is
$\G_2(a,Y)$.
\end{proof}

\begin{proof}[Proof of Theorem~\ref{thm:HM}]
The case $k=2$ is Lemma~\ref{lem:rank-two-five}; the equality family is
$\G_2(a,Y)$ and has size five.  Assume henceforth that $k\ge3$.  The
hypothesis \eqref{eq:HM-explicit-threshold} gives
\[
 b=n+1-k\ge3\left(m^3+\binom m2\right)\ge3,
\]
so $n\ge k+2$.  Hence both canonical candidates in
Lemma~\ref{lem:canonical-HM-candidates} are defined, and all counting
formulas in that lemma apply.

Let $\C$ and $\R$ be supplied by Lemma~\ref{lem:rank-two-peeling}.  We first
derive a small-core alternative.  If $\C$ has no common atom, then
$|\C|\le5\le m-1$, and hence
\[
 |\A|\le5E_2+|\R|
 \le(m-1)E_2+\Psi_k(b)E_2.
\]

Suppose instead that $\C$ has common atom $a$.  Since $\A$ is nontrivial,
choose $C\in\A$ with $a\nleq C$.  Lemma~\ref{lem:rank-two-peeling} gives
$|\A\setminus\A[a]|\le|\R|\le\Psi_k(b)E_2$.  Form
$\Hcal=\Hcal(a,C)$ as in~\eqref{eq:localized-high-family}.  If
$|\Hcal|\le m-1$, Lemma~\ref{lem:localized-high-covers} yields
\begin{equation}
\label{eq:HM-small-core-upper}
 |\A|\le(m-1)E_2+m^3E_3+\Psi_k(b)E_2.
\end{equation}
The same bound also holds in the preceding no-common-atom case.  If the small
alternative does not hold, then $|\Hcal|=m$.

We first exclude the small alternative for a maximum family.  The portion of
$\G_k(a,Y)$ through $a$ is the union of the $m$ full rank-two stars centred
at $a\vee e$, where $e$ ranges over the atoms of $C_0$.  These centres are
distinct.  The intersection of two corresponding stars consists of
extensions of their join, whose rank is at least three, and therefore has
size at most $E_3$.  The first Bonferroni inequality now gives
\[
 |\G_k(a,Y)|
 \ge \sum_{e\in\tau(C_0)}E_2
      -\sum_{\{e,f\}\in\binom{\tau(C_0)}2}E_3+2
 =mE_2-\binom m2E_3+2.
\]
Consequently~\eqref{eq:HM-small-core-upper},
Lemma~\ref{lem:exact-extension-ratio}, and
Lemma~\ref{lem:numerical-remainder} imply
\begin{align*}
 |\G_k(a,Y)|-|\A|
 &\ge E_2-\left(m^3+\binom m2\right)E_3-\Psi_k(b)E_2+2\\
 &\ge
 \left(1-\frac{m^3+\binom m2}{b}-\Psi_k(b)\right)E_2+2>0.
\end{align*}

It remains to consider $|\Hcal|=m$.  Write
$\Hcal=\{B_1,\ldots,B_m\}$.  Lemma~\ref{lem:localized-high-covers} shows
that every $B_i$ intersects every member of $\A$, and
Lemma~\ref{lem:core-classification} applies to $a$, the $B_i$, and $C$.

Let $A\in\A$.  If $a\le A$, then $A$ intersects $C$, because $C\in\A$.  If
$a\nleq A$, then $A$ intersects every $B_i$ by
Lemma~\ref{lem:localized-high-covers}; the final assertion of
Lemma~\ref{lem:core-classification} therefore classifies $A$.  In case (1)
we obtain
\[
  \A\subseteq\mathscr F_k(a,X),
\]
whereas in case (2) we obtain
\[
  \A\subseteq\G_k(a,Y).
\]
Here it is immaterial which of the two exceptional clique flats was chosen as
$C$: among flats containing $a$, meeting one is equivalent, by closure, to
meeting the other.
By Lemma~\ref{lem:canonical-HM-candidates}, the first construction is smaller than
the second by one.  This proves the upper bound, and equality forces case
(2) and $\A=\G_k(a,Y)$.  Conversely, $\G_k(a,Y)$ is a nontrivial
intersecting family of the asserted size.  All choices of $a$ and $Y$ are
equivalent under an automorphism of $K_{n+1}$, proving uniqueness up to
isomorphism.
\end{proof}

\begin{corollary}
Let $k\ge3$.  If a nontrivial intersecting family $\A$ is contained in
neither $\mathscr F_k(a,X)$ nor $\G_k(a,Y)$, for any admissible choice of
the parameters, then
\[
 |\A|\le (m_k-1)E_2+m_k^3E_3+\Psi_k(b)E_2.
\]
Under~\eqref{eq:HM-explicit-threshold}, this bound is strictly smaller than
$\operatorname{HM}(n,k)$.
\end{corollary}

\begin{proof}
The proof of Theorem~\ref{thm:HM} gives either the small-core bound
\eqref{eq:HM-small-core-upper} or containment in one of the two canonical
configurations.  The final assertion is the numerical comparison made in
that proof.
\end{proof}

\section{The linear Erd\H{o}s--Ko--Rado theorem}
\label{sec:linear-ekrproof}

The proof has four stages.  First, a rank-sensitive extension ratio is
combined with spread peeling to control the part of the family whose members
have bounded atomic weight.  Second, the complementary dense part is encoded
by an occupancy model; a modal exponential tilt gives both a global dense
budget and a tail estimate uniform in the atomic weight.  Third, a
singleton-switching argument compares every dense member outside a fixed
edge-star with missing members inside that star.  Finally, the sparse and
dense estimates are combined according to whether the sparse part has a
common atom.  This yields the threshold $N=n+1\ge8k$.  The combinatorial and
probabilistic reductions are proved below; the elementary rational estimates
needed to close the constants are recorded in
Appendix~\ref{app:technical-estimates} so that the main line of the argument
remains visible.

Put
\[
 N:=n+1,\qquad b=N-k,
 \qquad \omega(F):=|\tau(F)|
 =\sum_{B\in F}\binom{|B|}{2}.
\]
Thus a rank-$k$ flat is a $b$-block partition of an $N$-set, and
$E_i=\sbinom{N-i-1}{k-i}$ is the number of rank-$k$ extensions of a fixed
rank-$i$ flat.  We first record a ratio estimate that retains the fact that
the total rank is only $k$.  This is substantially sharper in the present
range than the general estimate $E_i/E_{i-1}\le1/b$.

\begin{lemma}
\label{lem:linear-stirling-ratio}
If $m>b\ge1$ and $s=m-b$, then
\begin{equation}
\label{eq:linear-stirling-ratio}
 \frac{\sbinom{m-2}{s-1}}{\sbinom{m-1}{s}}
 \le \frac{2(m-b)}{mb}.
\end{equation}
Consequently, for $i\ge2$,
\begin{equation}
\label{eq:linear-extension-product}
 \frac{E_i}{E_1}
 \le
 \prod_{j=1}^{i-1}\frac{2(k-j)}{(N-j)b}.
\end{equation}
\end{lemma}

\begin{proof}
Double-count pairs $(P,v)$ in which $P$ is a partition of an $m$-set into
$b$ blocks and $v$ belongs to a nonsingleton block of $P$.  Deleting $v$
and then inserting it into one of the $b$ blocks of a partition of the
remaining $m-1$ points shows that the number of pairs is
$mb\sbinom{m-2}{s-1}$.  On the other hand, a $b$-block partition
has at most $2s$ points in nonsingleton blocks: every block of size
$d\ge2$ satisfies $d\le2(d-1)$.  Hence the same number is at most
$2s\sbinom{m-1}{s}$, which proves~\eqref{eq:linear-stirling-ratio}.  Apply this
successively with $m=N-1,N-2,\ldots,N-i+1$ to obtain
\eqref{eq:linear-extension-product}.
\end{proof}

The next lemma is the part of the argument that uses spread approximation.
The important point is to peel after truncating by atomic weight.  Every
centre created by the peeling procedure is a subflat of an original member,
so the same atomic-weight bound remains valid at every stage.

\begin{lemma}
\label{lem:bounded-atom-peeling}
Let $\C\subseteq\F_k(M_n)$ be an intersecting family with no common atom.
If $\omega(F)\le\mu$ for every $F\in\C$, then
\begin{equation}
\label{eq:bounded-atom-peeling}
 |\C|\le \mu E_2+
 \sum_{i=2}^k\binom{i+1}{2}(\mu+1)^{i-1}E_i.
\end{equation}
\end{lemma}

\begin{proof}
Apply Lemma~\ref{lem:uniform-parameter-peeling} to $\C$ with $t=1$,
$q=\mu$, and $r=\mu+1$.  Since $\sbinom{i}{1}=m_i=\binom{i+1}{2}$, every
terminal layer satisfies
\[
 |\W_i|\le\binom{i+1}{2}(\mu+1)^{i-1}.
\]
The cover invariant therefore bounds the contribution of all rank-$i$
layers, $i\ge2$, by the sum in~\eqref{eq:bounded-atom-peeling}.

It remains to control the final rank-one core.  Since it is an intersecting
antichain of atoms, it is either empty or consists of one atom $a$.  In the
second case choose $F_0\in\C$ with $a\nleq F_0$, which is possible because
$\C$ has no common atom.  Every member $F\in\C[a]$ intersects $F_0$ in an
atom $e\in\tau(F_0)$.  Necessarily $e\ne a$, and $a\vee e\le F$.  There are
at most $\omega(F_0)\le\mu$ choices for $e$, and every centre $a\vee e$ has
rank two.  Hence
\[
 |\C[a]|\le\mu E_2.
\]
Adding this possible core contribution to the peeled layers proves
\eqref{eq:bounded-atom-peeling}.
\end{proof}

\begin{corollary}
\label{cor:linear-sparse-budget}
Suppose that $N\ge8k$ and $\mu+1\le7b/10$.  Under the hypotheses of
Lemma~\ref{lem:bounded-atom-peeling},
\[
 \frac{|\C|}{E_1}
 \le \frac7{40}+\left(1-\frac7{40}\right)^{-3}-1
<\frac{22}{23}.
\]
\end{corollary}

\begin{proof}
Put $r=\mu+1$.  Lemma~\ref{lem:linear-stirling-ratio} gives
\[
 \frac{\mu E_2}{E_1}
 \le\frac{2\mu(k-1)}{(N-1)b}
 \le\frac75\frac{k}{N}\le\frac7{40},
\]
and, for every $i\ge2$,
\[
 r^{i-1}\frac{E_i}{E_1}
 \le\prod_{j=1}^{i-1}\frac{2r(k-j)}{(N-j)b}
 \le\left(\frac75\frac{k}{N}\right)^{i-1}
 \le\left(\frac7{40}\right)^{i-1}.
\]
Substitute these bounds into~\eqref{eq:bounded-atom-peeling} and extend the
finite sum to infinity.  The identity
\[
 \sum_{i\ge2}\binom{i+1}{2}x^{i-1}=(1-x)^{-3}-1
\]
then gives the claimed bound.
\end{proof}

We next show that flats outside this sparse regime form a very small part of
the full level.  Set
\[
 r_0:=\left\lfloor\frac{7b}{10}\right\rfloor,
 \qquad
 T(r):=|\{F\in\F_k(M_n):\omega(F)\ge r\}|,
 \qquad
 \eta:=\frac{N-3k+1}{N-3k+2}.
\]

\begin{lemma}
\label{lem:linear-total-star-ratio}
One has
\[
 \frac{E_0}{E_1}
 =\frac{\sbinom{N-1}{k}}{\sbinom{N-2}{k-1}}
 \le\frac{N(N-1)}{2k}<\frac{N^2}{2k}.
\]
\end{lemma}

\begin{proof}
Double-count pairs $(F,a)$ with $F\in\F_k(M_n)$ and
$a\in\tau(F)$. Every atom belongs to exactly $E_1$ rank-$k$ flats, so
\[
 \sum_F\omega(F)=\binom N2E_1.
\]
If the blocks of $F$ have sizes $d_1,d_2,\ldots$, then
\[
 \omega(F)=\sum_i\binom{d_i}{2}
 \ge\sum_i(d_i-1)=k.
\]
Thus $kE_0\le\binom N2E_1$.
\end{proof}

To estimate $T(r)$, order the $b$ blocks and write their sizes as
\[
 1+y_1,\ldots,1+y_b,
 \qquad y_i\ge0,\qquad \sum_{i=1}^b y_i=k.
\]
For a fixed vector $y$, the number of ordered partitions is
$N!/\prod_i(y_i+1)!$. Hence, under the uniform distribution on
$\F_k(M_n)$,
\begin{equation}
\label{eq:linear-occupancy-law}
 \mathbb P(y_1,\ldots,y_b)\ \propto\
 \prod_{i=1}^b\frac1{(y_i+1)!},
 \qquad
 \omega(F)=\sum_{i=1}^b\binom{y_i+1}{2}.
\end{equation}
Let
\[
 f(z):=\sum_{j\ge0}\frac{z^j}{(j+1)!}=\frac{e^z-1}{z},
 \qquad c_s:=[z^s]f(z)^b.
\]
The sequence $(1/(j+1)!)_{j\ge0}$ is log-concave, and convolution preserves
log-concavity~\cite{Hoggar}; therefore $(c_s)_{s\ge0}$ is log-concave.

\begin{lemma}
\label{lem:linear-modal-tilt}
If $0<k<b$, there is a real number $\lambda>0$ for which $k$ is a mode of
the distribution
\begin{equation}
\label{eq:linear-tilted-sum}
 \mathbb P_\lambda(S=s):=\frac{c_s\lambda^s}{f(\lambda)^b}.
\end{equation}
Moreover, $\lambda$ may be chosen so that
\begin{equation}
\label{eq:linear-lambda-upper}
 \lambda\le\frac{2(k+1)}{b-k}.
\end{equation}
\end{lemma}

\begin{proof}
Log-concavity gives
\[
 \frac{c_{k-1}}{c_k}\le\frac{c_k}{c_{k+1}}.
\]
Choose $\lambda=c_k/c_{k+1}$.  Then the two adjacent ratios show that $k$
is a mode of~\eqref{eq:linear-tilted-sum}.

For the upper bound, write
$w(y)=\prod_i1/(y_i+1)!$ for $\sum_i y_i=s$. Increasing one coordinate
and using $w(y+e_i)=w(y)/(y_i+2)$ yields
\[
 \sum_{\sum y_i=s}w(y)\sum_{i=1}^b\frac1{y_i+2}
 =\sum_{\sum z_i=s+1}\nu(z)w(z),
\]
where $\nu(z)$ is the number of positive coordinates of $z$. The left side
is at least $(b-s)c_s/2$, because $y$ has at least $b-s$ zero coordinates,
whereas the right side is at most $(s+1)c_{s+1}$. Therefore
\[
 \frac{c_s}{c_{s+1}}\le\frac{2(s+1)}{b-s}.
\]
Taking $s=k$ gives the asserted bound on $\lambda$.
\end{proof}

For the estimates below we use the elementary exponential bounds
\begin{equation}
\label{eq:linear-elementary-exp}
 e^x\le \varphi(x):=1+x+\frac{x^2}{2(1-x/3)}
 \qquad(0\le x<3)
\end{equation}
and
\begin{equation}
\label{eq:linear-geometric-exp}
 e^x\le\frac1{1-x}
 \qquad(0\le x<1).
\end{equation}
The function $\varphi$ is increasing on $[0,3)$, and in particular
$e\le\varphi(1)=11/4<3$.  These facts are verified in
Lemma~\ref{lem:appendix-exp-bounds}.

We treat small and large $k$ separately. For bounded $k$, large
atomic weight forces a large block and a direct union bound is strongest.

\begin{lemma}
\label{lem:linear-dense-small}
Suppose that $N\ge8k$ and $1\le k<120$. For every $r\ge r_0$,
\begin{equation}
\label{eq:linear-dense-small}
 \frac{T(r)}{E_1\eta^r}<\frac1{25000}.
\end{equation}
\end{lemma}

\begin{proof}
If $k\le8$, then
\[
 r_0\ge\frac{49k}{10}-1>\binom{k+1}{2},
\]
so $T(r)=0$. Assume henceforth that $9\le k<120$. From
\eqref{eq:linear-occupancy-law},
\[
 \omega(F)=\frac12\sum_i y_i(y_i+1)
 \le\frac{k}{2}\bigl(1+\max_i y_i\bigr).
\]
Thus $\omega(F)\ge r$ forces a block of size at least
\[
 d:=\left\lceil\frac{2r}{k}\right\rceil\ge10.
\]
Selecting $d$ points in such a block and contracting them gives
\[
 T(r)\le\binom Nd E_{d-1}.
\]
If $d>k+1$, then $T(r)=0$.  Otherwise,
Lemma~\ref{lem:linear-stirling-ratio} gives
\begin{equation}
\label{eq:linear-Rd}
 \frac{T(r)}{E_1}
 \le R_d(N,k):=
 \frac{2^{d-2}N(N-d+1)(k-1)\cdots(k-d+2)}
      {d!\,b^{d-2}}.
\end{equation}
Lemma~\ref{lem:appendix-small-tail-estimate} proves the remaining uniform
estimate
\[
 R_d(N,k)\eta^{-r}<\frac1{25000}.
\]
Combining this with~\eqref{eq:linear-Rd} proves
\eqref{eq:linear-dense-small}.
\end{proof}

For $k\ge120$ we use a conditional exponential moment in the occupancy
model.

\begin{lemma}
\label{lem:linear-local-mass}
Suppose that $N\ge8k$ and $k\ge120$. Choose $\lambda$ as in
Lemma~\ref{lem:linear-modal-tilt}. Then
\begin{equation}
\label{eq:linear-local-mass}
 \mathbb P_\lambda(S=k)>\frac1{7\sqrt{k}}.
\end{equation}
\end{lemma}

\begin{proof}
By~\eqref{eq:linear-lambda-upper},
\begin{equation}
\label{eq:linear-lambda-rational}
 \lambda\le\frac{2(k+1)}{N-2k}
 \le\frac{121}{360},
 \qquad
 \frac{b\lambda}{k}
 \le2\left(1+\frac1k\right)\frac{N-k}{N-2k}
 \le\frac{847}{360}.
\end{equation}
Let $J_1,\ldots,J_b$ be independent with
\[
 \mathbb P_\lambda(J_i=j)
 =\frac{\lambda^j}{(j+1)!f(\lambda)}.
\]
Then $S=J_1+\cdots+J_b$ has distribution
\eqref{eq:linear-tilted-sum}. Direct summation gives
\[
 \mathbb E J_i^2
 =\frac{e^\lambda(\lambda^2-\lambda+1)-1}{e^\lambda-1}.
\]
The first estimate in Lemma~\ref{lem:appendix-local-mass-estimates}, applied
with $0<\lambda\le121/360$, gives
\[
 \mathbb E J_i^2\le\frac23\lambda.
\]
Consequently, if $\sigma^2=\operatorname{Var}(S)$, then
\[
 \sigma^2\le b\mathbb E J_i^2
 \le\frac{847}{540}k<\frac85k.
\]
Chebyshev's inequality places at least $3/4$ of the mass in an interval
containing at most $4\sigma+2$ integers. Since $k$ is a mode,
\[
 \mathbb P_\lambda(S=k)
 \ge\frac3{4(4\sigma+2)}
 \ge\frac3{4(4\sqrt{8k/5}+2)}.
\]
The second estimate in Lemma~\ref{lem:appendix-local-mass-estimates} shows
that the last expression is strictly larger than $1/(7\sqrt{k})$, proving
\eqref{eq:linear-local-mass}.
\end{proof}

\begin{lemma}
\label{lem:linear-dense-large}
Suppose that $N\ge8k$ and $k\ge120$. For every $r\ge r_0$,
\begin{equation}
\label{eq:linear-dense-large}
 \frac{T(r)}{E_0}
 \le7\sqrt{k}\,N^{161/120}N^{-6r/N}.
\end{equation}
\end{lemma}

\begin{proof}
Set $\theta=6\log N/N$. For $0\le s\le\theta$, let
\[
 g_s(\lambda):=\sum_{j=0}^k u_j(s),
 \qquad
 u_j(s):=\frac{\lambda^j e^{s \binom{j+1}{2}}}{(j+1)!}.
\]
The ratio estimates in
Lemma~\ref{lem:appendix-tilted-generating-function} give, uniformly for
$0\le s\le\theta$,
\[
 \sum_{j=1}^k \binom{j+1}{2}u_j(s)<\frac{19}{25}\lambda.
\]
Since $g_s(\lambda)\ge1$, it follows that
\[
 \frac{d}{ds}\log g_s(\lambda)
 =\frac{\sum_j \binom{j+1}{2}u_j(s)}{g_s(\lambda)}
 <\frac{19}{25}\lambda.
\]
As $g_0(\lambda)\le f(\lambda)$, integration from $0$ to $\theta$ yields
\begin{equation}
\label{eq:linear-g-ratio}
 \frac{g_\theta(\lambda)}{f(\lambda)}
 \le e^{(19/25)\lambda\theta}.
\end{equation}

Conditioning the independent occupancy variables on $S=k$ gives exactly
the law in~\eqref{eq:linear-occupancy-law}.  Moreover, $S=k$ implies
$J_i\le k$ for every $i$, so
$\mathbf 1_{\{S=k\}}\le\prod_i\mathbf 1_{\{J_i\le k\}}$.  Positivity and
independence therefore give
\begin{align*}
 \mathbb E[e^{\theta\omega(F)}]
 &=\frac{\mathbb E_\lambda\!\left[
      e^{\theta\sum_i \binom{J_i+1}{2}}\mathbf 1_{\{S=k\}}\right]}
        {\mathbb P_\lambda(S=k)}\\
 &\le \mathbb P_\lambda(S=k)^{-1}
       \left(\frac{g_\theta(\lambda)}{f(\lambda)}\right)^b\\
 &<7\sqrt{k}\,e^{(19/25)b\lambda\theta}
 <7\sqrt{k}\,N^{161/120}.
\end{align*}
Here the penultimate inequality uses
Lemma~\ref{lem:linear-local-mass} and~\eqref{eq:linear-g-ratio}, while the
final exponent estimate is the last assertion of
Lemma~\ref{lem:appendix-tilted-generating-function}.

Finally, $T(r)/E_0$ is the probability that a uniformly chosen rank-$k$
flat has atomic weight at least $r$.  Markov's inequality and
$e^{-\theta r}=N^{-6r/N}$ now give~\eqref{eq:linear-dense-large}.
\end{proof}

\begin{corollary}
\label{cor:linear-dense-budget}
If $N\ge8k$, then
\begin{equation}
\label{eq:linear-dense-budget}
 T(r_0)<\frac1{29}E_1.
\end{equation}
Moreover, for every $r\ge r_0$,
\begin{equation}
\label{eq:linear-uniform-tail}
 T(r)<E_1\eta^r.
\end{equation}
\end{corollary}

\begin{proof}
For $k<120$, both claims follow immediately from
Lemma~\ref{lem:linear-dense-small}. Suppose that $k\ge120$. Define
\[
 A(r):=\frac{E_0}{E_1}\,7\sqrt{k}\,N^{161/120}N^{-6r/N}.
\]
Lemma~\ref{lem:linear-dense-large} gives $T(r)/E_1\le A(r)$.  The endpoint
and monotonicity estimates in Lemma~\ref{lem:appendix-dense-budget-endpoints}
give
\[
 A(r_0)<\frac1{29}
\]
and, for every $r\ge r_0$,
\[
 A(r)\eta^{-r}\le A(r_0)\eta^{-r_0}<1.
\]
The first display proves~\eqref{eq:linear-dense-budget}; the second gives
\[
 \frac{T(r)}{E_1\eta^r}\le A(r)\eta^{-r}<1,
\]
which is~\eqref{eq:linear-uniform-tail}.
\end{proof}

The last ingredient compares a star with the members outside it. It is a
problem-specific switching argument rather than a generic consequence of
spreadness.

\begin{lemma}
\label{lem:linear-singleton-switching}
Let $N\ge3k$, let $a$ be an atom, and let $F\in\F_k(M_n)$ avoid $a$.
Then the set
\[
 \Gamma_a(F):=
 \{Q\in\E_k(a):\tau(Q)\cap\tau(F)=\varnothing\}
\]
satisfies
\begin{equation}
\label{eq:linear-switching-bound}
 |\Gamma_a(F)|\ge E_1
 \left(\frac{N-3k+1}{N-3k+2}\right)^{\omega(F)}
 =E_1\eta^{\omega(F)}.
\end{equation}
\end{lemma}

\begin{proof}
Contract the endpoints of $a$. Since $F$ avoids $a$, those endpoints lie in
different blocks of $F$, and $\tau(F)$ becomes a simple graph $G_F$ on
$N-1$ vertices with exactly $\omega(F)$ edges. Every original vertex has
degree at most $k$ inside its block. The contracted vertex has degree equal
to the sum of the ranks of the two original blocks, and this is also at most
$k$. Thus $\Delta(G_F)\le k$.

Members of $\E_k(a)$ become $b$-block partitions of the contracted
$(N-1)$-set. Such a partition belongs to $\Gamma_a(F)$ exactly when all its
blocks are independent in $G_F$. Add the edges of $G_F$ one at a time. Let
$H$ be the graph formed by the edges already added and let $uv$ be the next
edge. Every $H$-proper $b$-block partition has deficiency $k-1$, and hence
at least
\[
 (N-1)-2(k-1)=N-2k+1
\]
singleton blocks.

Consider an $H$-proper partition in which $u$ and $v$ lie in the same
block. At most $k$ singleton vertices are adjacent to $v$ in $H$, leaving
at least $\ell:=N-3k+1$ permissible singleton blocks $\{w\}$. Move $v$ out
of its old block and replace $\{w\}$ by $\{v,w\}$. The resulting partition
is still $H$-proper and separates $u$ from $v$. This map is injective on
pairs consisting of the original partition and $w$: in the image, $w$ is
the unique partner of $v$ in its two-point block.

If $X$ and $Y$ denote the numbers of $H$-proper partitions that put $u,v$
together and apart, respectively, then $Y\ge\ell X$. Adding the edge $uv$
therefore retains a proportion at least
\[
 \frac{Y}{X+Y}\ge\frac{\ell}{\ell+1}=\eta.
\]
Iterating over all $\omega(F)$ edges proves
\eqref{eq:linear-switching-bound}.
\end{proof}

\begin{proof}[Proof of Theorem~\ref{thm:EKR-linear-t1}]
Set
\[
 \A_{\mathrm{sp}}:=\{F\in\A:\omega(F)<r_0\}.
\]
If $\A_{\mathrm{sp}}\ne\varnothing$, its maximum atomic weight $\mu$ satisfies
$\mu+1\le r_0\le7b/10$.

Suppose first that $\A_{\mathrm{sp}}$ is nonempty and has no common atom. By
Corollaries~\ref{cor:linear-sparse-budget} and~\ref{cor:linear-dense-budget},
\[
 |\A|\le|\A_{\mathrm{sp}}|+T(r_0)
 <\left(\frac{22}{23}+\frac1{29}\right)E_1
 =\frac{661}{667}E_1<E_1.
\]
If $\A_{\mathrm{sp}}=\varnothing$, the dense budget alone gives the same strict
conclusion.

It remains to consider the case in which $\A_{\mathrm{sp}}$ has a common atom
$a$. Put
\[
 \C:=\A\cap\E_k(a),
 \qquad
 \B:=\A\setminus\E_k(a).
\]
Every member of $\B$ has atomic weight at least $r_0$.  If
$\B=\varnothing$, then
$\A\subseteq\E_k(a)$ and hence $|\A|\le E_1$.

Assume that $\B\ne\varnothing$, and choose $F\in\B$ minimizing
$\omega(F)$; write $r=\omega(F)$. Then $r\ge r_0$ and
$|\B|\le T(r)$. Corollary~\ref{cor:linear-dense-budget} and
Lemma~\ref{lem:linear-singleton-switching} give
\[
 |\B|\le T(r)<E_1\eta^r\le|\Gamma_a(F)|.
\]
No member of $\Gamma_a(F)$ can belong to $\C$, because $\A$ is
intersecting. Consequently
\[
 |\C|\le E_1-|\Gamma_a(F)|<E_1-|\B|,
\]
and $|\A|<E_1$.

We have proved the desired upper bound. Equality can occur only when
$\B=\varnothing$ and $\A\subseteq\E_k(a)$; it then forces
$\A=\E_k(a)$. Conversely, every full edge-star is intersecting and has
size $E_1$.
\end{proof}

\section{Further problems and conjectures}
\label{sec:further}

We discuss the remaining questions in the same order as the main results.
For $t=1$, Theorem~\ref{thm:EKR-linear-t1} replaces the quadratic
hypothesis of Theorem~\ref{thm:EKR} by $n+1\ge8k$.  The sharp target remains
Czabarka's conjecture, Conjecture~\ref{conj:sharp-EKR-t1}, namely $n\ge2k$.
Thus the present argument settles the correct order and leaves only the
constant-factor gap between $n+1\ge8k$ and $n+1\ge2k+1$.

For $t=1$, the bounded-atom truncation and occupancy tail
remove the quadratic atomic loss and give a linear threshold.  The remaining
barrier is located in the final switching step.

\begin{remark}[Limit of the present linear method]
\label{rem:linear-method-limit}
The constant $8$ is the point at which the estimates in this paper are
closed, but the structural obstruction in the same proof framework occurs
at the smaller constant $3$.  Indeed, in
Lemma~\ref{lem:linear-singleton-switching}, an $H$-proper partition has at
least $N-2k+1$ singleton blocks, while at most $k$ of their vertices are
forbidden by adjacency to the vertex being moved.  Thus the switching has
only
\[
 \ell=(N-2k+1)-k=N-3k+1
\]
guaranteed choices and retains the proportion $\ell/(\ell+1)$ at each step.
This proportion is useful precisely when $\ell\ge1$, equivalently
$N\ge3k$.  Consequently, sharper calculation in the bounded-atom and
occupancy estimates may lower the constant $8$, but, as long as the final
comparison uses this singleton switching, it can at best reach the range
$N\ge3k$ and cannot enter $N\le3k-1$.  This is a limitation of the method,
not a claim that the preceding numerical estimates already prove the
theorem for every $N\ge3k$.  Reaching Czabarka's conjectured range
$N\ge2k+1$ requires a new replacement for the switching step, such as a
global closure-compatible matching or compression.
\end{remark}

For $k>t\ge1$, let $n_0(k,t)$ denote the least integer such that the
conclusion of Theorem~\ref{thm:EKR} holds for every $n\ge n_0(k,t)$.
We next record a higher-intersection obstruction that is independent of the
Hilton--Milner theorem.

\begin{corollary}
\label{cor:EKR-core-obstruction}
Let $t\ge2$, $k=t+1$, and $n=t+3=2k-t+1$.  Then there is a
$t$-intersecting family in $\F_k(M_n)$ of size $2t+3$, whereas a full
$t$-star has size $6$.  In particular, the conclusion of
Theorem~\ref{thm:EKR} need not hold at $n=2k-t+1$.

More generally, a necessary condition for the EKR bound when $k=t+1$ is
\[
 \binom{n+1-t}{2}\ge2t+3.
\]
For the full conclusion including uniqueness of equality, the inequality
must be strict.  Thus
\[
 n_0(t+1,t)\ge
 t+\left\lfloor\frac{1+\sqrt{16t+25}}2\right\rfloor.
\]
\end{corollary}

\begin{proof}
Let $W$ be a $(t+3)$-set and choose a trivalent tree with leaf set $W$,
that is, a tree in which every non-leaf vertex has degree three.  Such a
tree has $2(t+3)-3=2t+3$ edges.  Deleting an edge gives an unordered
bipartition $\sigma=\{W_1,W_2\}$ of $W$ into two nonempty parts; let
$Q_\sigma$ be the flat whose only
nonsingleton blocks are $W_1$ and $W_2$, with every vertex outside $W$
left singleton.  Then
\[
 \rk(Q_\sigma)=(|W_1|-1)+(|W_2|-1)=t+1=k.
\]
For two edges of a tree, one component of the first deletion is disjoint from
one component of the second deletion.  Hence the common refinement of two distinct corresponding two-block partitions has exactly
three nonempty blocks on $W$, and therefore
\[
 \rk(Q_\sigma\wedge Q_{\sigma'})=|W|-3=t.
\]
Thus the $2t+3$ flats $Q_\sigma$ form a $t$-intersecting family.  Moreover,
each leaf edge of the tree gives the singleton split
$\{x\}\mid(W\setminus\{x\})$.  The common refinement of these singleton
splits is discrete on $W$, so the family has common meet of rank zero and is
not a full $t$-star.

At $n=t+3$ a full $t$-star has size
\[
 E_t=\sbinom{3}{1}=6.
\]
The same compatible-split construction embeds for every larger $n$ by
adding singleton vertices, while for $k=t+1$ a full $t$-star has size
\[
 \sbinom{n-t}{1}=\binom{n+1-t}{2}.
\]
The final lower bound is the least integral solution of the corresponding
strict inequality.
\end{proof}

At the level of the proof, a failed general-$t$ simplification is witnessed
by one of at most $m_i=\binom{i+1}{2}$ atoms of a rank-$i$ flat, which
produces the quadratic scale in Theorem~\ref{thm:EKR}.

\begin{problem}
Determine the least $n_0(k,t)$ for which the conclusion of
Theorem~\ref{thm:EKR} holds.
\end{problem}

The explicit Hilton--Milner threshold in Theorem~\ref{thm:HM} is
$n=O(k^6)$.  It is natural to ask whether the correct threshold is linear.

\begin{problem}
Determine the least $n_0^{\mathrm{HM}}(k)$ for which the conclusion of
Theorem~\ref{thm:HM} holds.  In particular, decide whether
$n_0^{\mathrm{HM}}(k)=O(k)$.
\end{problem}
In the Hilton--Milner argument, constructing $m_k+1$ disjoint petals above
a rank-two cover costs a factor $m_k^2$, and summing over all possible covers
creates the principal loss.

\bigskip
\noindent\textbf{Declaration of AI usage.}
The authors used generative-AI tools during preliminary exploration and
language editing.  The authors verified all mathematical statements and
proofs and take full responsibility for the contents of the manuscript.

\clearpage
\appendix
\section{Technical estimates for the EKR bounds}
\label{app:technical-estimates}

This appendix contains only the one-variable inequalities and finite
rational estimates used in the EKR proofs.  All combinatorial reductions,
the occupancy representation, the exponential tilting argument, and the
switching step remain in Section~\ref{sec:linear-ekrproof}.  We give the details
here to keep the explicit constants fully verifiable without obscuring the
proof structure in the main text.

\subsection{Elementary exponential estimates}

\begin{lemma}
\label{lem:appendix-exp-bounds}
The inequalities \eqref{eq:linear-elementary-exp} and
\eqref{eq:linear-geometric-exp} hold.  Moreover, $\varphi$ is increasing on
$[0,3)$ and $e<3$.
\end{lemma}

\begin{proof}
For every integer $m\ge2$,
\[
 m!\ge2\cdot3^{m-2}.
\]
Hence, for $0\le x<3$,
\[
 e^x
 =1+x+\sum_{m\ge2}\frac{x^m}{m!}
 \le1+x+\frac{x^2}{2}\sum_{j\ge0}\left(\frac{x}{3}\right)^j
 =\varphi(x),
\]
which proves~\eqref{eq:linear-elementary-exp}.  For $0\le x<1$, the
coefficientwise bound $1/m!\le1$ gives
\[
 e^x=\sum_{m\ge0}\frac{x^m}{m!}
 \le\sum_{m\ge0}x^m=\frac1{1-x},
\]
proving~\eqref{eq:linear-geometric-exp}.  Finally,
\[
 \varphi'(x)=1+\frac{3x(6-x)}{2(3-x)^2}>0
 \qquad(0\le x<3),
\]
so $\varphi$ is increasing, and
$e\le\varphi(1)=11/4<3$.
\end{proof}

\subsection{The dense tail for bounded rank}

\begin{lemma}
\label{lem:appendix-small-tail-estimate}
Suppose that $N\ge8k$, $9\le k<120$, $r\ge r_0$, and
\[
 d=\left\lceil\frac{2r}{k}\right\rceil\le k+1.
\]
Then the quantity $R_d(N,k)$ defined in~\eqref{eq:linear-Rd} satisfies
\[
 R_d(N,k)\eta^{-r}<\frac1{25000}.
\]
\end{lemma}

\begin{proof}
For fixed $k,d,r$, the expression $R_d(N,k)\eta^{-r}$ decreases for
$N\ge8k$ and $d\ge10$.  Indeed, the logarithmic derivative of its only
nonconstant algebraic factor satisfies
\[
 \frac1N+\frac1{N-d+1}<\frac{d-2}{N-k},
\]
because $d\le k+1$ gives
\[
 \frac1N+\frac1{N-d+1}
 \le\frac2{N-k}<\frac{d-2}{N-k}.
\]
Also, $\eta=(N-3k+1)/(N-3k+2)$ increases with $N$.  We may therefore set
$N=8k$.  Since $d=\lceil2r/k\rceil$, we have $r\le dk/2$.  Put
\[
 H_d(k):=R_d(8k,k)
 \left(\frac{5k+2}{5k+1}\right)^{dk/2}.
\]
Then $R_d(N,k)\eta^{-r}\le H_d(k)$.  For $10\le d\le k$, direct division
of consecutive terms gives
\begin{align*}
 \frac{H_{d+1}(k)}{H_d(k)}
 &=\frac{2(8k-d)(k-d+1)}{7k(d+1)(8k-d+1)}
   \left(1+\frac1{5k+1}\right)^{k/2}\\
 &\le\frac{2}{7(d+1)}e^{1/10}
 \le\frac{20}{63(d+1)}<1,
\end{align*}
where $1+x\le e^x$ and~\eqref{eq:linear-geometric-exp} were used.  Thus
$H_d(k)\le H_{10}(k)$, also when $d=k+1$, and
\[
 H_{10}(k)
 <\frac{2^{14}e}{10!\,7^8}k^2
 <\frac{2^{14}\cdot3\cdot120^2}{10!\,7^8}
 =\frac{4096}{121060821}<\frac1{25000}.
\]
This proves the claim.
\end{proof}

\subsection{A local-mass estimate for the modal tilt}

\begin{lemma}
\label{lem:appendix-local-mass-estimates}
If $0<\lambda\le121/360$, then
\[
 \frac{e^\lambda(\lambda^2-\lambda+1)-1}{e^\lambda-1}
 \le\frac23\lambda.
\]
Moreover, for every $k\ge120$,
\[
 \frac3{4(4\sqrt{8k/5}+2)}>\frac1{7\sqrt{k}}.
\]
\end{lemma}

\begin{proof}
For $0\le x\le121/360$, the polynomial $3-7x-3x^2$ is decreasing and
\[
 3-7\left(\frac{121}{360}\right)
 -3\left(\frac{121}{360}\right)^2=\frac{13319}{43200}>0.
\]
Also $3x^2-5x+3>0$, since its discriminant is $-11$.  Consequently,
\[
 \frac{3-2x}{3-5x+3x^2}-\varphi(x)
 =\frac{x^2(3-7x-3x^2)}
        {2(3-x)(3x^2-5x+3)}\ge0.
\]
Together with~\eqref{eq:linear-elementary-exp}, this gives
\[
 e^\lambda\le\frac{3-2\lambda}{3-5\lambda+3\lambda^2}.
\]
Therefore
\begin{align*}
 &3(e^\lambda-1)
 \left(\frac23\lambda-
 \frac{e^\lambda(\lambda^2-\lambda+1)-1}{e^\lambda-1}\right)\\
 &\qquad=3-2\lambda-e^\lambda(3-5\lambda+3\lambda^2)\ge0,
\end{align*}
which proves the first assertion.

For the second assertion, it is enough to show
\[
 16\sqrt{\frac85}+\frac8{\sqrt{k}}<21.
\]
Since $\sqrt{8/5}<19/15$ and, for $k\ge120$,
$8/\sqrt{k}\le8/\sqrt{120}<11/15$, the left side is smaller than
\[
 16\cdot\frac{19}{15}+\frac{11}{15}=21.
\]
\end{proof}

\subsection{The tilted generating function}

\begin{lemma}
\label{lem:appendix-tilted-generating-function}
Suppose that $N\ge8k$, $k\ge120$, put $b=N-k$, and let $\lambda$ satisfy
\eqref{eq:linear-lambda-rational}.  Set
\[
 \theta=\frac{6\log N}{N},
 \qquad
 u_j(s)=\frac{\lambda^j e^{s \binom{j+1}{2}}}{(j+1)!}
 \quad(0\le s\le\theta).
\]
Then
\begin{equation}
\label{eq:linear-theta-rational}
 \theta<\frac7{160},
\end{equation}
and, uniformly for $0\le s\le\theta$,
\[
 \sum_{j=1}^k \binom{j+1}{2}u_j(s)<\frac{19}{25}\lambda.
\]
Moreover,
\[
 \frac{19}{25}b\lambda\theta
 <\frac{161}{120}\log N.
\]
\end{lemma}

\begin{proof}
The function $x\mapsto\log x/x$ decreases for $x>e$.  Moreover,
\[
 e^{7/10}>1+\frac7{10}+\frac{(7/10)^2}{2}
 +\frac{(7/10)^3}{6}=\frac{12013}{6000}>2,
\]
so $\log2<7/10$ and
\[
 \theta\le\frac{6\log960}{960}
 <\frac{6\log1024}{960}<\frac7{160}.
\]
Consecutive terms satisfy
\[
 \frac{u_{j+1}(s)}{u_j(s)}
 =\frac{\lambda e^{s(j+1)}}{j+2}.
\]
The logarithm of the right side is convex in $j$, so its maximum for
$0\le j<k$ is attained at an endpoint.  At $j=0$,
\eqref{eq:linear-lambda-rational}, \eqref{eq:linear-theta-rational}, and
\eqref{eq:linear-geometric-exp} give
\[
 \frac{u_1(s)}{u_0(s)}
 \le\frac{121}{720}e^{7/160}
 \le\frac{121}{720}\frac{160}{153}<\frac12.
\]
At $j=k-1$,
\[
 \frac{u_k(s)}{u_{k-1}(s)}
 \le\frac{2N^{6k/N}}{N-2k}.
\]
Writing $x=N/k\ge8$, we have $6/x-1\le-1/4$, $k\ge120$,
$x-2\ge6$, and $x^{1/x}$ decreasing for $x\ge e$.  Hence
\[
 \frac{2N^{6k/N}}{N-2k}
 =\frac{2k^{6/x-1}x^{6/x}}{x-2}
 \le\frac{8^{3/4}}{3\,120^{1/4}}<\frac12,
\]
where the last inequality follows, after taking fourth powers, from
$8192<9720$.  Therefore
\begin{equation}
\label{eq:linear-geometric-u}
 \max_{0\le j<k}\frac{u_{j+1}(s)}{u_j(s)}<\frac12.
\end{equation}

We next sharpen the first six ratios.  Since $\varphi$ is increasing,
\eqref{eq:linear-elementary-exp} gives
\begin{align*}
 \frac{u_1(s)}{\lambda}
 &\le\frac12\varphi\left(\frac7{160}\right)
 =\frac{158129}{302720}<\frac{23}{44},\\
 \frac{u_2(s)}{u_1(s)}
 &\le\frac{121}{1080}\varphi\left(\frac7{80}\right)
 =\frac{182347}{1491200}<\frac6{49},\\
 \frac{u_3(s)}{u_2(s)}
 &\le\frac{121}{1440}\varphi\left(\frac{21}{160}\right)
 =\frac{750563}{7833600}<\frac1{10}.
\end{align*}
Also,
\[
 \lambda e^{6s}
 \le\frac{121}{360}\varphi\left(\frac{21}{80}\right)
 =\frac{1837627}{4204800}<\frac7{16}.
\]
Thus, for $3\le j\le5$,
\[
 \frac{u_{j+1}(s)}{u_j(s)}
 \le\frac7{16(j+2)}\le\frac1{j+8}.
\]
Multiplying these ratio bounds gives
\begin{equation}
\label{eq:linear-first-six-rational}
\begin{aligned}
 \sum_{j=1}^6 \binom{j+1}{2}u_j(s)
 &<\frac{23}{44}\lambda\Bigg(
 1+3\frac6{49}+6\frac6{49}\frac1{10}
 +10\frac6{49}\frac1{10}\frac1{11}\\
 &\hspace{30mm}
 +15\frac6{49}\frac1{10}\frac1{11}\frac1{12}
 +21\frac6{49}\frac1{10}\frac1{11}\frac1{12}\frac1{13}
 \Bigg)\\
 &=\frac{1171229}{1541540}\lambda,\\
 u_6(s)
 &<\frac{23}{44}\frac6{49}\frac1{10}\frac1{11}\frac1{12}\frac1{13}
 \lambda
 =\frac{23}{6166160}\lambda.
\end{aligned}
\end{equation}
By~\eqref{eq:linear-geometric-u},
\[
 \sum_{j=7}^k \binom{j+1}{2}u_j(s)
 <u_6(s)\sum_{m\ge1}\binom{m+7}{2}2^{-m}
 =37u_6(s),
\]
where
$\sum_{m\ge1}2^{-m}=1$, $\sum_{m\ge1}m2^{-m}=2$, and
$\sum_{m\ge1}m^2 2^{-m}=6$.  Combining this with
\eqref{eq:linear-first-six-rational}, we obtain
\[
 \sum_{j=1}^k \binom{j+1}{2}u_j(s)
 <\frac{4685767}{6166160}\lambda
 <\frac{19}{25}\lambda.
\]
Finally, \eqref{eq:linear-lambda-rational} and $N\ge8k$ give
\[
 \frac{19}{25}\frac{b\lambda}{k}\frac{6k}{N}
 \le\frac{19}{25}\frac{847}{360}\frac68
 =\frac{16093}{12000}<\frac{161}{120}.
\]
Multiplication by $\log N$ proves the final assertion.
\end{proof}

\subsection{Closing the dense budget}

\begin{lemma}
\label{lem:appendix-dense-budget-endpoints}
Suppose that $N\ge8k$ and $k\ge120$, and define
\[
 A(r)=\frac{E_0}{E_1}\,7\sqrt{k}\,N^{161/120}N^{-6r/N}.
\]
Then
\[
 A(r_0)<\frac1{29}.
\]
Moreover, $A(r)\eta^{-r}$ decreases with $r$, and
\[
 A(r_0)\eta^{-r_0}<\frac3{29}<1.
\]
\end{lemma}

\begin{proof}
Lemma~\ref{lem:linear-total-star-ratio} and
$r_0\ge7b/10-1\ge49N/80-1$ give
\begin{align*}
 A(r_0)
 &<\frac7{2\sqrt{k}}N^{401/120-6r_0/N}\\
 &\le\frac7{2\sqrt{k}}N^{-1/3}N^{6/N}\\
 &<\frac{147}{80}k^{-5/6}.
\end{align*}
By~\eqref{eq:linear-theta-rational} and
\eqref{eq:linear-geometric-exp},
\[
 N^{6/N}=e^\theta<\frac{160}{153}<\frac{21}{20},
\]
while $N^{1/3}\ge2k^{1/3}$.  Also $(9/4)^6>120$, so
$120^{5/6}>160/3$.  Since $k\ge120$,
\[
 A(r_0)
 <\frac{147}{80}\frac3{160}
 =\frac{441}{12800}<\frac1{29}.
\]

For monotonicity, put $D=N-3k+1$, so $\eta=D/(D+1)$.  Since
$D>5N/8$ and $N\ge960>e$,
\[
 \theta=\frac{6\log N}{N}>
 \frac6N>\frac8{5N}>\frac1D>
 \log\left(1+\frac1D\right).
\]
Thus the factor
\[
 N^{-6r/N}\eta^{-r}
 =\exp\left[-r\left(\theta-
       \log\left(1+\frac1D\right)\right)\right]
\]
decreases with $r$.  Finally,
\[
 \frac{r_0}{D}
 \le\frac{7(N-k)}{10(N-3k+1)}
 <\frac{7(N-k)}{10(N-3k)}\le\frac{49}{50}.
\]
Consequently,
\[
 A(r_0)\eta^{-r_0}
 <\frac1{29}e^{r_0/D}
 \le\frac{e^{49/50}}{29}<\frac3{29},
\]
where Lemma~\ref{lem:appendix-exp-bounds} gives $e<3$.
\end{proof}

\subsection{The explicit series estimate for the general EKR bound}

\begin{lemma}
\label{lem:appendix-Phi-bound}
Let $k>t\ge1$.  If $b\ge c_t(m_k+1)$, then
\[
 \sum_{i=t+1}^k \sbinom{i}{t}
       \left(\frac{r_i}{b}\right)^{i-t}\le\frac6{11},
 \qquad
 \frac{m_k}{b}<\frac1{36}.
\]
\end{lemma}

\begin{proof}
Put $Q=m_k+1$ and write $i=t+j$.  Since $r_i\le Q$ and
\eqref{eq:subflat-polynomial-bound} gives $\sbinom{i}{t}\le(i+1)^{2t}$, while
$i+1=t+j+1\le(t+2)j$ for $j\ge1$, we have
\[
 \sum_{i=t+1}^k \sbinom{i}{t}
       \left(\frac{r_i}{b}\right)^{i-t}
 \le(t+2)^{2t}\sum_{j\ge1}\frac{j^{2t}}{c_t^j}.
\]
For every integer $p\ge1$,
\[
 j^p\le p!\binom{j+p-1}{p},
\]
and the binomial series therefore gives
\[
 \sum_{j\ge1}\frac{j^{2t}}{c_t^j}
 \le\frac{(2t)!}{c_t}
       \left(1-\frac1{c_t}\right)^{-(2t+1)}.
\]
Bernoulli's inequality yields
\[
 \left(1-\frac1{c_t}\right)^{-(2t+1)}
 \le\left(1-\frac{2t+1}{c_t}\right)^{-1}.
\]
Since
\[
 \frac{2t+1}{(t+2)^{2t}(2t)!}\le\frac16
 \qquad\text{and}\qquad
 c_t=2(t+2)^{2t}(2t)!,
\]
the preceding estimates give
\[
 \sum_{i=t+1}^k \sbinom{i}{t}
       \left(\frac{r_i}{b}\right)^{i-t}
 \le\frac12\left(1-\frac1{12}\right)^{-1}
 =\frac6{11}.
\]
Finally, $b\ge c_t(m_k+1)>c_t m_k$, and
$(t+2)^{2t}(2t)!\ge18$, so $c_t\ge36$.  Hence
\[
 \frac{m_k}{b}<\frac1{c_t}\le\frac1{36}.
\]
\end{proof}

\end{document}